\documentclass[11pt]{amsart}
\pdfoutput=1 

\usepackage[
text={440pt,575pt},
headheight=9pt,
centering
]{geometry}

\usepackage{hyperref}
\usepackage{xcolor}

\definecolor{darkred}{RGB}{160,0,0}
\definecolor{darkblue}{RGB}{0,0,160}
\hypersetup{
  colorlinks,
  citecolor=darkblue,
  filecolor=black,
  linkcolor=darkblue,
  urlcolor=darkblue
}

\usepackage[utf8]{inputenc}
\usepackage[T1]{fontenc}
\usepackage{microtype}
\usepackage{lmodern}
\usepackage{amsmath,amssymb,amsthm}
\usepackage{soul}
\usepackage{mdwlist}
\usepackage[pdftex]{graphicx}
\usepackage{latexsym}
\usepackage{verbatim}
\usepackage{graphicx,caption,subcaption}
\usepackage{xcolor}
\usepackage{todonotes} 
\usepackage{enumitem}
\usepackage{upgreek}

\allowdisplaybreaks 

\theoremstyle{plain}
\newtheorem{theorem}{Theorem}[section]
\newtheorem{corollary}[theorem]{Corollary} 
\newtheorem{lemma}[theorem]{Lemma}
\newtheorem{proposition}[theorem]{Proposition}

\theoremstyle{definition}
\newtheorem{definition}[theorem]{Definition} 	
\newtheorem{remark}[theorem]{Remark}

\newcommand{\R}{\mathbb{R}}

\newcommand{\N}{\mathbb{N}}
\newcommand{\PP}{\mathbb{P}}
\newcommand{\NN}{\mathrm{N}}

\newcommand{\A}{\mathcal{A}}
\newcommand{\D}{\mathcal{D}}

\newcommand{\E}{\mathbb{E}}
\newcommand{\Var}{\mathrm{Var}}
\newcommand{\Cov}{\mathrm{Cov}}

\newcommand{\Bin}{\mathrm{Bin}}
\newcommand{\GR}{\mathrm{GR}}
\newcommand{\rk}{\mathrm{rk}}
\newcommand{\Inv}{\mathrm{Inv}}
\newcommand{\inv}{\mathrm{inv}}
\newcommand{\des}{\mathrm{des}}
\newcommand{\dx}{\mathrm{d}x}
\newcommand{\dy}{\mathrm{d}y}

\renewcommand{\P }{{\mathbb P}}
\usepackage{mathtools}
\newcommand{\eid}{\stackrel{\D}{=}}
\newcommand{\Hajek}{H\'{a}jek}
\newcommand{\vep}{\varepsilon}
\newcommand{\nto}{n \to \infty}
\newcommand{\kn}{k_n}

\title{Joint extremes of inversions and descents of random permutations}
\author[P. Dörr]{Philip Dörr}
\address{Department of Mathematics, Otto-von-Guericke University Magdeburg, Universitätsplatz 2, 39106 Magdeburg, Germany}
\email{philip.doerr@ovgu.de}
\author[J. Heiny]{Johannes Heiny}
\address{Department of Mathematics,
Stockholm University,
Albano hus 1,
10691 Stockholm,
Sweden}
\email{johannes.heiny@math.su.se}



\begin{document}

\begin{abstract}
    We provide asymptotic theory for the joint distribution of $X_\inv$ and $X_\des$, the numbers of inversions and descents of random permutations. Recently, \cite{dorr2022extreme} proved that $X_\inv$, respectively, $X_\des$ is in the maximum domain of attraction of the Gumbel distribution. To tackle the dependency between these two permutation statistics, we use Hájek projections and a suitable quantitative Gaussian approximation. We show that $(X_\inv, X_\des)$ is in the maximum domain of attraction of the two-dimensional Gumbel distribution with independent margins. This result can be stated in the broader combinatorial framework of finite Coxeter groups, on which our method also yields the central limit theorem for $(X_\inv, X_\des)$ and various other permutation statistics as a novel contribution. 
    In particular, signed permutation groups with random biased signs and products of classical Weyl groups are investigated.
\end{abstract}

\maketitle

\section{Introduction}

The numbers of inversions and descents are two of the most important characteristics of permutations. On the symmetric group $S_n$ of permutations $\pi \negmedspace: \{1, \ldots, n\} \rightarrow \{1, \ldots, n\},$ an inversion is any tuple $(i,j)$ with $i < j,$ but $\pi(i) > \pi(j)$. A descent is any inversion of two adjacent numbers, that is, any $i$ with $\pi(i) > \pi(i+1)$. Upon drawing permutations uniformly at random, these permutation statistics become random variables that we denote by $X_\inv$ and $X_\des$. The study of stochastic properties of quantities such as the number of inversions and descents belongs to the field of statistical algebra and is the aim of this paper. 
The asymptotic distribution of $X_\inv$ and $X_\des$ has been investigated by several authors (see, e.g., \cite{bender1973central, chatterjee2017central, fulman2004stein, harper1967stirling, kahle2020counting, pitman1997probabilistic}), who showed that inversions and descents each satisfy a central limit theorem (CLT) and thus are asymptotically normal. We say that a family of real-valued random variables $X_1, X_2, \ldots$  (or their respective distributions) satisfies the CLT if 
\[
    \frac{X_n - \E(X_n)}{\Var(X_n)^{1/2}} \overset{\D}{\longrightarrow} \NN(0,1)\,, \qquad n\to \infty\,,
\]
where $\overset{\D}{\longrightarrow}$ denotes convergence in distribution. For the number of inversions or descents on finite Coxeter groups, there are several techniques and approaches toward the CLT. Chatterjee \& Diaconis provide an overview of proofs of the CLT for $X_\des$ on the family of symmetric groups in \cite[Section~3]{chatterjee2017central}. One approach is based on a representation of $X_\des$ via $m$-dependent variables. Other proofs of the asymptotic normality of $X_\des$ are based on the zeros of its generating function \cite{harper1967stirling, pitman1997probabilistic} or on certain regularity properties of the generating function, see \cite[Ex.~3.5 and 5.3]{bender1973central}. Stein's method of exchangeable pairs has been applied in \cite{conger2005refinement, fulman2004stein}, where the latter reference and \cite[Ex. 5.5]{bender1973central} also cover inversions. We note that Stein's method has been used in various settings, see, e.g., \cite{conger2007normal} for permutations on multisets, \cite{pike2011convergence} for generalized inversions, and \cite{arslan2018unfair} for non-uniformly random permutations. 

The work of Dörr \& Kahle~\cite{dorr2022extreme} was the first to provide extreme value theory for these classical permutation statistics. They showed that the numbers of inversions and descents are in the maximum domain of attraction of the standard Gumbel distribution, assuming a triangular array where the number of samples per row obeys an exponential upper bound.  

While the asymptotic normality of inversions and descents as \textit{univariate statistics} is well studied, the knowledge of their \textit{joint distribution} is comparatively sparse. The dependency structure of $(X_\inv, X_\des)^\top$ poses a challenge, and many methods such as, e.g., Janson's dependency criterion \cite[Theorem 2]{janson1988normal} are not applicable in the multivariate setting. Fang \& Röllin \cite{fang2015rates} gave a CLT for arbitrarily large collections of permutation statistics based on antisymmetric matrices, including $(X_\inv, X_\des)^\top$ as a special case. Our novel approach is to replace $X_\inv$ with its H\'{a}jek projection $\hat{X}_\inv$, giving an $m$-dependent approximation $(\hat{X}_\inv, X_\des)^\top$. Then, a recent Gaussian approximation result for $m$-dependent random vectors by Chang \textit{et al.} \cite{chang2021central} can be used to obtain the asymptotic extreme value limit behavior of $(X_\inv, X_\des)^\top$. 

Symmetric groups belong to the class of finite Coxeter groups, on which the concept of inversions and descents can be readily generalized, see \cite[Section~1.4]{bjorner2006combinatorics}. As in \cite{dorr2022extreme}, we state our findings in this broader framework. Besides the family $(S_n)_{n \in \N}$ of symmetric groups, we also focus on the \textit{signed permutation groups} $B_n$ and the \textit{even-signed permutation groups} $D_n$. For these three families,  we use the umbrella term \textit{classical Weyl groups} throughout. For some basics on finite Coxeter groups, we refer to \cite{bjorner2006combinatorics}. Kahle \& Stump \cite{kahle2020counting} developed a full characterization of CLTs for inversions and descents on sequences of finite Coxeter groups, by giving necessary and sufficient conditions on the asymptotics of $\Var(X_\inv)$ and $\Var(X_\des)$. Additionally, we now obtain the CLT of $(X_\inv, X_\des)^\top$ on all classical Weyl groups as well as their finite products. 

This paper is structured as follows. Section~\ref{section2} introduces the H\'{a}jek projection of inversions and descents on symmetric groups and justifies its use to approximate $X_\inv$. Section~\ref{section3} presents the corresponding extreme value limit theorem (EVLT) of $(X_\inv, X_\des)^\top$ as the main result of this paper. In Section~\ref{section5}, these results are extended to the larger groups $B_n$ and $D_n$. We also equip these with a new family of probability measures, namely the so-called $p$-biased signed permutations. Section~\ref{section8} concludes the paper with some open questions. The remaining technical proofs are gathered in Section~\ref{section9}. \par
\vspace{1mm}
Throughout this paper, we denote the standard uniform distribution on the interval $[0,1]$ by $U(0,1)$, and the discrete uniform distribution on the set $\{0, 1, \ldots, n\}$ by $U(\{0,1,\ldots, n\})$. We sometimes also use these notations for accordingly distributed random variables when the meaning is clear from the context. The symbol $\eid$ means equality in distribution and $\sum_{i<j}$ will be used as an abbreviation of the double-indexed sum $\sum_{1\leq i<j\leq n}$. For any random variable $X$, we denote its standard deviation by $\sigma(X) = \sqrt{\Var(X)}$. The correlation of two random variables $X,Y$ is $\rho(X,Y)$. Moreover, we use typical Landau notation for positive sequences $a_n, b_n$ as follows: 
\begin{itemize}
    \item $a_n = O(b_n)$ means that $a_n$ grows at most as fast as~$b_n$, i.e.,
    $\limsup_{n\rightarrow\infty} a_n/b_n < \infty$. 
    \item $a_n = o(b_n)$ means that $a_n$ grows slower than $b_n$, i.e.,   $\lim_{n\rightarrow\infty} a_n/b_n = 0$. This is also written as $a_n \ll b_n$ or $b_n \gg a_n$.
    \item $a_n = \Theta(b_n)$ means that $a_n$ and $b_n$ have the same order of magnitude, i.e., both $a_n = O(b_n)$ and $b_n = O(a_n)$ hold.
    \item $a_n = b_n + o_\PP(1)$ means that $a_n, b_n$ are sequences of random variables with $a_n - b_n \overset{\PP}{\longrightarrow} 0$.
\end{itemize}

\section{H\'{a}jek projections on symmetric groups} \label{section2} 

A permutation drawn uniformly at random from the symmetric group $S_n$ of permutations on $n$ letters is induced by the ranks of independent random variables $Z_1, \ldots, Z_n \sim U(0,1)$. Thus, the numbers of inversions and descents of this permutation can be represented by
\begin{align}
    X_\inv &= \sum_{1 \leq i < j \leq n} \textbf{1}\{Z_i > Z_j\}, \label{1.1} \\
    X_\des &= \sum_{i=1}^{n-1} \textbf{1}\{Z_i > Z_{i+1}\}. \label{1.2} 
\end{align}

The primary challenge in dealing with the joint permutation statistic $(X_\inv,X_\des)$ is the dependence structure between $X_\inv$ and $X_\des$. The random variables $\textbf{1}\{Z_i > Z_j\}, 1\le i<j\le n$ in \eqref{1.1} are also dependent. It is worth noting that $X_\inv$ has the following representation as a sum of $n-1$ independent terms:
\begin{equation}
    X_\inv \eid \sum_{i=1}^{n-1} U(\{0, 1, \ldots, i\})\,. \label{1.3}
\end{equation}
This follows, e.g., from \cite[Corollary 2.5a)]{dorr2022extreme} within the framework of all finite Coxeter groups. The representation of $X_\des$ in \eqref{1.2} has $m$-dependent variables (precisely, $m=1$). There is also a decomposition of $X_\des$ into independent summands, based on the splitting of its generating function (known as the \textit{Eulerian polynomial}) into linear factors of its real-valued roots \cite[Corollary 2.5b)]{dorr2022extreme}. 

From the representations in \cite[Corollary 2.5a)]{dorr2022extreme} and \cite[Corollary 2.5b)]{dorr2022extreme} it is easy to derive the generating functions of $X_\inv$ and $X_\des$. That $X_\inv$ and $X_\des$ are not independent follows from the fact that the joint generating function of $(X_\des, X_\inv)$ does not factor into those of $X_\inv$ and $X_\des$. It is an interesting question whether this polynomial factors at all. Some tests were made on $S_3$ or $S_4,$ but there was no regularity detected so far. 

To handle the dependence between $X_\inv$ and $X_\des$, we approximate $X_\inv$ through its H\'{a}jek projection $\hat{X}_\inv$ and write $(\hat{X}_\inv, X_\des)$ as a sum of $m$-dependent two-dimensional random vectors. 

\begin{definition} \label{def2.1}
For independent random variables $Z_1, \ldots, Z_n$ and any random variable $X_n$, the \textit{H\'{a}jek projection} of $X_n$ with respect to $Z_1, \ldots, Z_n$ is given by
\[\hat{X}_n := \sum_{k=1}^n \E(X_n \mid Z_k) - (n-1)\E(X_n).\]
\end{definition}
Note that $\E(\hat{X}_n)=\E(X_n)$. Since each $\E(X_n \mid Z_k)$ is a measurable function only in $Z_k$, the H\'{a}jek projection is a sum of independent random variables, regardless of the original dependence structure between $X_n$ and $Z_k$. In Sections~\ref{section2} and \ref{section3}, we always assume that $Z_1,\ldots,Z_n \sim U(0,1)$ for our purposes. To decide whether the H\'{a}jek projection is a sufficiently accurate approximation, the following criterion is useful.

\begin{theorem}[cf.~\cite{van2000asymptotic}, Theorem 11.2] \label{thm2.2}
Consider a sequence $(X_n)_{n\ge 1}$ of random variables and their associated H\'{a}jek projections $(\hat{X}_n)_{n\ge 1}$. If $\Var(\hat{X}_n) \sim \Var(X_n)$ as $n \rightarrow \infty,$ then
\[
    \frac{X_n - \E(X_n)}{\Var(X_n)^{1/2}} = \frac{\hat{X}_n - \E(\hat{X}_n)}{\Var(\hat{X}_n)^{1/2}} + o_{\PP}(1).
\]
\end{theorem} 
In particular, if $\Var(\hat{X}_n) \sim \Var(X_n)$ and $(\hat{X}_n)_{n\ge 1}$ satisfies a CLT, then Theorem~\ref{thm2.2} guarantees that $(X_n)_{n\ge 1}$ also satisfies a CLT.

In what follows, for a random variable or vector $X$ with finite variance, we write $Y$ for its standardization, that is, $Y=\bigl(X-\E(X)\bigr)/\sqrt{\Var(X)}$. In particular, $Y_\inv$ is the standardization of $X_\inv$ and $\hat{Y}_\inv$ is that of $\hat{X}_\inv$. We use the symbols $X_\inv, Y_\inv, \hat{X}_\inv, \hat{Y}_\inv, X_\des, Y_\des$ with suppression of $n$, the underlying symmetric group or its rank, unless needed for clarification.
\smallskip

The next result provides the \Hajek\ projection of $X_\inv$ defined in \eqref{1.1} and verifies the variance equivalence condition stated in Theorem~\ref{thm2.2} for $X_\inv$.

\begin{lemma} \label{lemma2.4}
The H\'{a}jek projection $\hat{X}_\inv$ of $X_\inv$ is given by 
\[\hat{X}_\inv = \frac{n(n-1)}{4} + \sum_{k=1}^n (n-2k+1)Z_k\]
and it holds that $\Var(X_\inv) \sim \Var(\hat{X}_\inv)$ as $n\to \infty$.
\end{lemma}
\begin{figure}[t]
    \centering
    \includegraphics[scale=1.2]{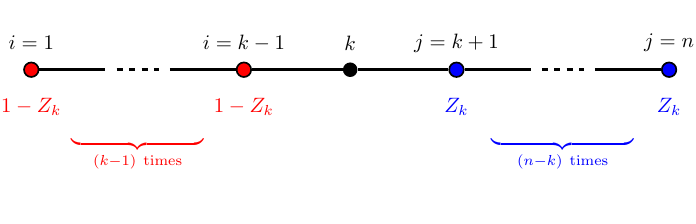}  \vspace{-7.5mm}
    \caption{Display of the non-constant contributions to $\E(X_\inv \mid Z_k)$}
    \label{fig1}
\end{figure}

\begin{proof} 
We first study the conditional expectations $\E(X_\inv \mid Z_k)$ for $k = 1, \ldots, n$ and get
\[
    \E(X_\inv \mid Z_k) = \sum_{1 \leq i<j \leq n} \PP(Z_i > Z_j \mid Z_k) = \sum_{1 \leq i<j \leq n} \begin{cases} 1/2, & \text{ if } k \notin \{i,j\}, \\ Z_k, & \text{ if } k=i, \\ 1 - Z_k, & \text{ if } k=j. \end{cases} 
\]
We fix $k \in \{1, \ldots, n\}$ and analyze the frequency of the three cases on the right-hand side. As $\{1, \ldots, n\} \setminus \{k\}$ has cardinality $n-1,$ there are $\binom{n-1}{2}$ subsets $\{i,j\} \subseteq \{1, \ldots, n\} \setminus \{k\}$. In case of $i=k,$ there are $n-k$ indices $j$ with $j > k,$ for which we have $\PP(Z_k > Z_j \mid Z_k) = \PP(Z_j < Z_k \mid Z_k) = Z_k$ since $Z_k \sim U(0,1)$. Likewise, in case of $j=k,$ there are $k-1$ indices $i$ with $i < k,$ which gives $\PP(Z_i > Z_k \mid Z_k) = 1 - Z_k$. These contributions are illustrated in Figure~\ref{fig1}. Therefore, we obtain
\begin{align*}
    \E(X_\inv \mid Z_k) &= \frac{1}{2}\binom{n-1}{2} + (n-k)Z_k + (k-1)(1-Z_k) \\
    &= \frac{1}{2}\binom{n-1}{2} + (n-2k+1)Z_k + (k-1),
\end{align*}
from which we deduce that
\begin{align}
    \hat{X}_\inv &= \sum_{k=1}^n \E(X_\inv \mid Z_k) - (n-1)\E(X_\inv) \nonumber \\
    &= \frac{n}{2}\binom{n-1}{2} + \sum_{k=1}^n (n-2k+1)Z_k + \sum_{k=1}^n (k-1) - \frac{n-1}{2}\binom{n}{2} \nonumber \\
    &= \frac{n(n-1)}{4} + \sum_{k=1}^n (n-2k+1)Z_k. \label{5.1}
\end{align}

\noindent Since the $Z_k$'s are i.i.d.\,, the variance of the H\'{a}jek projection is  
$$\Var(\hat{X}_\inv) = \sum_{k=1}^n \Var((n - 2k + 1)Z_k).$$
Due to $\Var(Z_k) = 1/12,$ we get
\begin{align*}
    \Var(\hat{X}_\inv) &= \frac{1}{12}\sum_{k=1}^n (2k - n - 1)^2 
    = \frac{1}{12} \sum_{k=1}^n \bigl(4k^2 + (n+1)^2 - 4k(n+1)\bigr) \\
    &= \frac{1}{12} \left(4\sum_{k=1}^n k^2 + n(n+1)^2 - 4(n+1)\frac{n(n+1)}{2}\right) \\
    &= \frac{1}{12} \left(4\frac{n(n+1)(2n+1)}{6}-n(n+1)^2 \right) = \frac{1}{36}n^3  -\frac{n}{36}\,, \qquad n\to \infty.
\end{align*}
By \cite[Corollary 3.2]{kahle2020counting}, we have $\Var(X_\inv) = \displaystyle{\frac{1}{36} n^3 +\frac{9n^2+7n}{72}}$ and therefore $\Var(X_\inv) \sim \Var(\hat{X}_\inv)$ as $n\to \infty$.
\end{proof}

A combination of Theorem~\ref{thm2.2} and Lemma~\ref{lemma2.4} yields 
$$Y_\inv = \hat{Y}_\inv + o_{\PP}(1)\,, \qquad n\to \infty\,.$$

\begin{remark} \label{rem2.5}
Interestingly, this approach fails for $X_\des,$ since $\Var(X_\des) \sim \Var(\hat{X}_\des)$ does not hold. 
Repeating the considerations in the proof of Lemma~\ref{lemma2.4} for $X_\des,$ we first obtain
\[
    \E(X_\des \mid Z_k) = \sum_{i=1}^{n-1} \PP(Z_i > Z_{i+1} \mid Z_k) = \sum_{i=1}^{n-1} \begin{cases} 1/2, & k \notin \{i,i+1\}, \\ Z_k, & k=i, \\ 1 - Z_k, & k=i+1. \end{cases}
\]
Now, except for the boundary cases $k=1$ and $k=n,$ the summands for $k=i$ and $k=i+1$ are each used exactly once, so the $Z_k$ in their sum $Z_k+(1 - Z_k)$ cancel out. In total, we obtain 
\[
    \hat{X}_\des = Z_1 - Z_n + c_n\,,
\]
where $c_n$ is some constant that depends only on $n$. Therefore, $ \Var(\hat{X}_\des) = 2/12$ does not have the linear order of $\Var(X_{\des}) = (n+1)/12$ (see \cite[Corollary 4.2]{kahle2020counting}). 
\end{remark}
For these reasons, our results will be based on the following consequence of Theorem~\ref{thm2.2} and Lemma~\ref{lemma2.4}.

\begin{corollary} \label{cor2.5} 
Let $(X_\inv, X_\des)^\top$ be given from the symmetric group $S_n$. For the standardized random vector $(Y_\inv, Y_\des)^\top$ and the standardized H\'{a}jek projection $\hat{Y}_\inv$, we have
\[
    \begin{pmatrix} Y_\inv \\ Y_\des \end{pmatrix} = \begin{pmatrix} \hat{Y}_\inv + o_{\PP}(1) \\ Y_\des \end{pmatrix}.
\]
A decomposition of $(\hat{X}_\inv, X_\des)^\top$ into $1$-dependent summands is given by
\begin{align}\label{eq:1}
\begin{pmatrix} \hat{X}_\inv  \\ X_\des  \end{pmatrix} &= \sum_{k=1}^{n-1} \begin{pmatrix} (n-2k+1)Z_k \\ \mathbf{1}\{Z_k > Z_{k+1}\} \end{pmatrix} +  \begin{pmatrix} -(n-1)Z_n+\frac{n(n-1)}{4} \\ 0 \end{pmatrix}.
\end{align}
Likewise, a $1$-dependent decomposition for $(\hat{Y}_\inv, Y_\des)^\top$ can be found by standardization.
\end{corollary}

It is worth noting that the correlation $\rho(X_\inv, X_\des)$ is not zero. However, we now show that $\rho(X_\inv, X_\des)\to 0$ as $n\to \infty$. Moreover, by Corollary~\ref{cor2.5} the same holds true for $(\hat{X}_\inv, X_\des)^{\top}$ as well (in fact, $\rho(\hat{X}_\inv, X_\des)$ is even easier to compute). To proceed, we need the covariance of $\hat{X}_\inv$ and $X_\des$. Our next result additionally provides $\Cov(X_\inv, X_\des)$, which -- to the best of our knowledge -- is not available in the literature.  

\begin{lemma}[see Subsections~\ref{sec9.1}, ~\ref{sec9.2} for the proof] \label{lemma2.6}
On the symmetric group $S_n$, we have 
\begin{enumerate}
\item[\emph{a)}] $\Cov(X_\inv, X_\des) = (n-1)/4$.
\item[\emph{b)}] $\Cov(\hat{X}_\inv, X_\des)=(n-1)/6$.
\end{enumerate}
\end{lemma}

\begin{corollary} \label{cor2.7}
Since $\Var(X_\inv)\Var(X_\des) = \Theta(n^4)$ according to \cite[Corollaries 3.2 and 4.2]{kahle2020counting}\emph{,} and the same holds true if $\Var(X_\inv)$ is replaced by $\Var(\hat{X}_\inv),$ we conclude from Lemma \ref{lemma2.6} that 
\begin{align*}
    \rho(X_\inv, X_\des) &= \dfrac{\Cov(X_\inv, X_\des)}{\sqrt{\Var(X_\inv)\Var(X_\des)}} = \Theta(1/n)\,, \\
    \rho(\hat{X}_\inv, X_\des) &= \Theta(1/n)\,, \qquad n \rightarrow \infty.
\end{align*}
\end{corollary}

\section{The bivariate Extreme Value Limit Theorem} \label{section3}

In this section, we establish the extreme value limit behavior of $(X_\inv, X_\des)^{\top}$ by using the $1$-dependent decomposition of $(\hat{X}_\inv, X_\des)^{\top}$ and applying a recent (and quite optimized) CLT for $m$-dependent triangular arrays from Chang \textit{et al.} \cite{chang2021central}. Their work provides Gaussian approximations for high-dimensional data under various dependency frameworks, including $m$-dependence. It gives error rates over the system of all hyperrectangles, including the Kolmogorov distance as a special case. Moreover, the high-dimensional framework implicitly covers the finite-dimensional one by repeating the components of a vector. 

Refining the notation of \cite{chang2021central}, we consider triangular arrays $(X_t^{(n)})_{t=1, \ldots, n}$ whose entries $X_1^{(n)}, \ldots, X_n^{(n)}$ are mean zero random vectors in $\R^\mathfrak{p}$, where $\mathfrak{p} = \mathfrak{p}(n)$ can grow with respect to $n$. For the sum 
\begin{equation}
    X^{(n)} := \sum_{t=1}^n X_t^{(n)} \quad \text{ with covariance matrix } \Sigma^{(n)} := \Var(X^{(n)}), \label{eq4}
\end{equation}
the work of Chang \textit{et al.} \cite{chang2021central} gives bounds and rates of convergence for 
\[
    r_n(\A^{\mathrm{re}}) := \sup_{A \in \A^{\mathrm{re}}} |\PP(X^{(n)} \in A) - \PP(\mathcal{N}_n \in A)|, 
\]
where $\A^{\mathrm{re}} := \Bigl\{\{\textbf{w} \in \R^\mathfrak{p} \negmedspace: \textbf{a} \leq \textbf{w} \leq \textbf{b}\} \mid \textbf{a}, \textbf{b} \in [-\infty, \infty]^\mathfrak{p}\Bigr\}$ is the system of all hyperrectangles, and $\mathcal{N}_n \sim \NN(0, \Sigma^{(n)})$ is a normal distribution with the same covariance structure as $X^{(n)}$. Bounds for $\A^{\mathrm{re}}$ (or for other systems of convex sets) in both constant and high dimensions have been strongly investigated for \textit{independent} variables. In recent years, there have been great efforts to improve the error bounds and the growth of dimension within the independent framework, see \cite{chernozhukov2013gaussian, chernozhukov2017central, chernozhukov2023nearly, das2021central, fang2021high, koike2021notes}. 

An interesting feature of \cite{chang2021central} is that the $X_t^{(n)}$ are allowed to be dependent which offers a wide range of applications beyond the independent framework, including the $1$-dependent decomposition of $(\hat{X}_\inv, X_\des)^{\top}$. The following two conditions need to be imposed on the $X_{t}^{(n)}= (X_{t,1}^{(n)}, \ldots,X_{t,\mathfrak{p}}^{(n)})^{\top}$. \par 
\bigskip
\noindent \textbf{Condition 1:} There exists a sequence of constants $B_n \geq 1$ and a universal constant $\gamma_1 \geq 1$ such that 
\[
    \max_{j=1,\ldots, \mathfrak{p}}\E\left(\exp\left(\left|\sqrt{n} X_{t,j}^{(n)}\right|^{\gamma_1} B_n^{-\gamma_1}\right)\right) \leq 2\,, \qquad t=1,\ldots,n.
\]
\textbf{Condition 2:} There exists a constant $K > 0$ such that for all $n \in \N \negmedspace$
\[
    \min_{j=1,\ldots, \mathfrak{p}} \Var\left(\sum_{t=1}^n X_{t,j}^{(n)}\right) \geq K.
\]
\begin{remark}
Condition 1 means sub-Gaussianity, i.e., by Markov's inequality, 
\[
    \forall u > 0 \negmedspace: \quad \PP(|\sqrt{n}X_{t,j}^{(n)}| > u) \leq 2\exp(-u^{\gamma_1}B_n^{-\gamma_1}).
\]
For sub-Gaussian variables, particularly for the bounded variables $X_\des, X_\inv$ and the H\'{a}jek projection $\hat{X}_\inv$, we can choose $\gamma_1 = 2$ and $B_n = O(1)$. Condition 2 implies non-degeneracy, which obviously holds true in our setting. 
\end{remark}

\begin{proposition}[see \cite{chang2021central}, Corollary 1] \label{thm3.1}
Let $(X_t^{(n)})_{t=1,\ldots,n}$ be a triangular array of mean zero random vectors in high dimensions, i.e., $X_1^{(n)}, \ldots, X_n^{(n)} \in \R^\mathfrak{p}$ with $\mathfrak{p} = \mathfrak{p}(n) \gg n^\kappa$ for a constant $\kappa > 0$. Assume that each row $X_1^{(n)}, \ldots, X_n^{(n)}$ is $m$-dependent with a global constant $m \in \N$. Under Conditions 1 and 2, it holds that
\[
    r_n(\A^{\mathrm{re}}) = O\left(\frac{B_nm^{2/3}\log(\mathfrak{p})^{7/6}}{n^{1/6}}\right)\,, \qquad n\to \infty.
\]
\end{proposition}
We note that if $\mathfrak{p}$ remains constant, we can artificially repeat the vector components (say $n$ times) and therefore, the requirement $\mathfrak{p} \gg n^\kappa$ can be removed. We obtain the following corollary. 

\begin{corollary} \label{cor3.1}
Let $(X_t^{(n)})_{t=1,\ldots,n}$ be a triangular array of mean zero random vectors in fixed dimension $\mathfrak{p}$ and suppose that each row $X_1^{(n)}, \ldots, X_n^{(n)}$ is $m$-dependent with a global constant $m \in \N$. Under Conditions 1 and 2, it holds that 
\[
    r_n(\A^{\mathrm{re}}) = O\left(\frac{B_nm^{2/3}\log(n)^{7/6}}{n^{1/6}}\right)\,, \qquad n\to \infty\,.
\]
\end{corollary}

In what follows, we use the Gaussian approximation of Proposition~\ref{thm3.1} to prove that the vector of componentwise maxima of i.i.d.\ copies of $(X_\inv, X_\des)^\top$ is attracted to the bivariate Gumbel distribution with independent margins. Let $N = (N^{(1)}, N^{(2)})$ be a bivariate normal distribution with $N^{(1)}, N^{(2)} \sim \NN(0,1)$ and $\rho(N^{(1)}, N^{(2)}) < 1$, and let $(\kn)_{n \in \N}$ be a divergent sequence of positive integers. It is known by \cite[Theorem 3]{sibuya1960bivariate} that the maximum of $k_n$ i.i.d.\ samples $(N_i)_{i=1,\ldots,k_n}$ is attracted toward the bivariate standard Gumbel distribution with independent margins by virtue of  
\begin{equation}
  \alpha_n \big(\max\{N_1, \ldots, N_{\kn}\} - \alpha_n\big) \overset{\D}{\longrightarrow}
  \Lambda_2\,, \qquad \nto\,, \label{eqlambda}
\end{equation}
where $\Lambda_2(x,y) = \exp(-\exp(-x-y))$ and
\begin{align*}
     \alpha_n := \sqrt{2\log \kn} - \frac{\log \log \kn + \log(4\pi)}{2 \sqrt{2\log \kn}}\,,
\end{align*}
and \eqref{eqlambda} is read component-wise. This can be used together with Proposition~\ref{thm3.1}, but in our application, we first need to replace $X_\inv$ with its \Hajek\ projection $\hat{X}_\inv$, and then we need to control the error caused by this replacement.
	
Let $(X_\inv^{(j)}, X_\des^{(j)})^\top, j=1,\ldots, k_n$ be independent copies of $(X_\inv, X_\des)^\top$ on $S_n$. We are interested in the asymptotic joint distribution of the component-wise maxima of $(X_\inv, X_\des)^\top$. Equivalently, we investigate the standardized maxima
\begin{align}
    M_{n,\inv} &:= \frac{\max_{j=1,\ldots,\kn} X_\inv^{(j)} - \E(X_\inv)}{\sigma(X_\inv)} \quad \text{ and } \quad
    M_{n,\des} := \frac{\max_{j=1,\ldots,\kn} X_\des^{(j)} - \E(X_\des)}{\sigma(X_\des)}\,. \label{eq:minv}
\end{align}
We now postulate the main result of this paper. If the number of samples $k_n$ is not too large, then the distribution of $(X_\inv, X_\des)^\top$ is in the maximum domain of attraction of a bivariate Gumbel distribution with independent margins. 

\begin{theorem} \label{thm4.3}
Consider the setting from above and assume $(\kn \log \kn)/n\to 0$ as $\nto$. Then, it holds that
\begin{align}\label{eq:mainevlt}
\lim_{\nto}\P\big(\alpha_n (M_{n,\inv} -\alpha_n) \le x, \alpha_n (M_{n,\des} -\alpha_n) \le y  \big) = \Lambda(x) \Lambda(y)\,, \qquad x,y\in \R\,.
\end{align}
In particular, the maxima of inversions and descents on $S_n$ are asymptotically independent.
\end{theorem}

\begin{proof}
Let $(Z_i^{j})_{i,j\ge 1}$ be a collection of independent $U(0,1)$ distributed random variables and recall that $\alpha_n\sim \sqrt{2 \log \kn}$. Then we have the representation
\[
    \begin{pmatrix} X_\inv^{(j)} \\ X_\des^{(j)} \end{pmatrix}_{j=1,\ldots,\kn} \eid \begin{pmatrix}
    \sum_{1 \leq i<k \leq n} \textbf{1}\{Z_i^{(j)} > Z_k^{(j)}\} \\ \sum_{i=1}^{n-1} \textbf{1}\{Z_i^{(j)} > Z_{i+1}^{(j)}\} \end{pmatrix}_{j=1,\ldots,\kn}. 
\] 
Therefore, by Slutsky's theorem, \eqref{eq:mainevlt} is an immediate consequence of 
\begin{align}\label{eq:mainevlt2}
    \lim_{\nto}\P\big(\alpha_n (\hat{M}_{n} -\alpha_n) \le x, \alpha_n (M_{n,\des} -\alpha_n) \le y  \big) = \Lambda(x) \Lambda(y)\,, \qquad x,y\in \R\,,
\end{align}
and 
\begin{equation}\label{eq:dfdgs}
\sqrt{\log \kn} \,|M_{n,\inv}-\hat{M}_{n}| \overset{\mathbb{P}}{\longrightarrow} 0\,, \qquad \nto\,,
\end{equation}
where $\hat{M}_{n} := \sigma(\hat{X}_\inv)^{-1}\left(\max_{j=1,\ldots,\kn} \hat{X}_\inv^{(j)} - \E(X_\inv)\right)$.
It remains to show \eqref{eq:mainevlt2} and \eqref{eq:dfdgs}. We begin with the proof of \eqref{eq:dfdgs} and get
\begin{align*}
    |M_{n,\inv}-\hat{M}_{n}| &\le \max_{j=1,\ldots,\kn} \bigg|   \frac{X_\inv^{(j)} - \E(X_\inv)}{\sigma(X_\inv)}- \frac{\hat{X}_\inv^{(j)} - \E(X_\inv)}{\sigma(\hat{X}_\inv)}\bigg|\\
    &= \max_{j=1,\ldots,\kn} \bigg|   \frac{X_\inv^{(j)}-\hat{X}_\inv^{(j)}}{\sigma(X_\inv)} +(\hat{X}_\inv^{(j)} - \E(X_\inv)) \frac{\sigma(\hat{X}_\inv)-\sigma(X_\inv)}{\sigma(X_\inv)\sigma(\hat{X}_\inv)} \bigg|\,.
\end{align*}
Thus, for any $\vep>0$, we obtain
\begin{align*} 
    &\P\Big( \sqrt{\log \kn} \,|M_{n,\inv}-\hat{M}_{n}|>2 \vep\Big)\le \P\Big( \sqrt{\log \kn} \max_{j=1,\ldots,\kn}\bigg|   \frac{X_\inv^{(j)}-\hat{X}_\inv^{(j)}}{\sigma(X_\inv)} \bigg|>\vep \Big)\\
    &+ \P\Big( \sqrt{\log \kn} \max_{j=1,\ldots,\kn}\bigg|   (\hat{X}_\inv^{(j)} - \E(X_\inv)) \frac{\sigma(\hat{X}_\inv)-\sigma(X_\inv)}{\sigma(X_\inv)\sigma(\hat{X}_\inv)} \bigg|>\vep \Big)=:P_1+P_2\,.
\end{align*}
Using the union bound and Markov's inequality, we have
\begin{align}
    P_1&\le \kn \,\P \Big(  |X_\inv-\hat{X}_\inv |> \frac{\sigma(X_\inv) \vep}{\sqrt{\log \kn}} \Big)
    \le \kn\, \frac{\log \kn}{\Var(X_\inv) \vep^2} \E |X_\inv-\hat{X}_\inv |^2 \notag\\
    &= \frac{\kn \log \kn}{\Var(X_\inv)\vep^2} \Big( \Var(X_\inv)+\Var(\hat{X}_\inv)-2 \Cov(X_\inv,\hat{X}_\inv)  \Big) \notag\\
    &= \frac{\kn \log \kn}{\vep^2} \Big( 1-\frac{\Var(\hat{X}_\inv)}{\Var(X_\inv)}  \Big)\,, \label{p1}
\end{align}
where the last equality follows from the fact that $\Cov(X_\inv,\hat{X}_\inv)=\Var(\hat{X}_\inv)$, see, e.g., \cite[Theorem~11.1]{van2000asymptotic}. Plugging in the formulas for $\Var(X_\inv)$ and $\Var(\hat{X}_\inv)$ from the end of the proof of Lemma~\ref{lemma2.4}, we get $\Var(\hat{X}_\inv)/\Var(X_\inv)=1+O(1/n)$ from which we conclude that 
$$P_1=\kn \log \kn \, O(1/n)\,,\qquad \nto\,,$$
which tends to zero by the assumption on $\kn$. Repeating the above considerations for $P_2$ yields
\begin{align}
P_2&\le \frac{\kn  \log \kn}{\vep^2} \Big(\frac{\sigma(\hat{X}_\inv)-\sigma(X_\inv)}{\sigma(X_\inv)}\Big)^2 \E\Big(\frac{\hat{X}_\inv^{(j)} - \E(X_\inv)}{\sigma(\hat{X}_\inv)}\Big)^2 \notag\\
&= \kn \log \kn \, O(1/n)=o(1)\,,\qquad \nto\,, \label{p2}
\end{align}
which completes the proof of \eqref{eq:dfdgs}. Regarding \eqref{eq:mainevlt2}, we recall that by Corollary~\ref{cor2.5}, 
\begin{align*}
\begin{pmatrix} \hat{X}_\inv^{(j)} -\E(\hat{X}_\inv^{(j)})  \\ X_\des^{(j)}-\E(X_\des^{(j)}) \end{pmatrix} &= 
\sum_{k=1}^{n-1} \begin{pmatrix} (n-2k+1)(Z_k^{(j)}-1/2) \\ \mathbf{1}\{Z_k^{(j)} > Z_{k+1}^{(j)}\} -1/2 \end{pmatrix} +  \begin{pmatrix} -(n-1)(Z_n^{(j)}-1/2) \\ 0 \end{pmatrix}.
\end{align*}
This is a sum of 1-dependent centered random vectors. Setting
\begin{align*} 
Y_k^{(n,j)}&:= \begin{pmatrix} (n-2k+1)(Z_k^{(j)}-1/2)/\sigma(\hat{X}_\inv) \\ (\mathbf{1}\{Z_k^{(j)} > Z_{k+1}^{(j)}\} -1/2)/\sigma(X_\des) \end{pmatrix}, \qquad k=1,\ldots,n-1,\\
Y_n^{(n,j)}&:= \begin{pmatrix} -(n-1)(Z_n^{(j)}-1/2)/\sigma(\hat{X}_\inv) \\ 0 \end{pmatrix},
\end{align*}
we obtain the representation $(\hat{Y}_\inv^{(j)}, Y_\des^{(j)})^\top = \sum_{k=1}^n Y_k^{(n,j)}$. 
The covariance matrix of $(\hat{Y}_\inv^{(j)}, Y_\des^{(j)})^\top$  is given by $\Sigma^{(n)}=\begin{pmatrix} 1 & \rho_n \\ \rho_n & 1 \end{pmatrix}$, where $\rho_n:= \rho(\hat{X}_\inv, X_\des)$. 
For a centered normal random vector $\mathcal{N}_{n}=(N_1, \ldots,N_{2\kn})^\top$ whose covariance matrix is block-diagonal with all $k_n$ diagonal blocks equal to $\Sigma^{(n)}$, 
we write
$$P_n(x,y):=\P\left(\alpha_n \Big(\max_{j=1,\ldots, \kn} N_{2j-1} -\alpha_n\Big) \le x, \alpha_n \Big(\max_{j=1,\ldots, \kn} N_{2j} -\alpha_n\Big) \le y  \right)\,, \qquad x,y\in \R\,.$$
We can also write
\begin{align*}
    &\P\big(\alpha_n (\hat{M}_{n} -\alpha_n) \le x, \alpha_n (M_{n,\des} -\alpha_n) \le y\big) \\
    =~& \P\left(\alpha_n(\hat{Y}_\inv^{(1)}, \ldots, \hat{Y}_\inv^{(k_n)})^\top - \boldsymbol{\upalpha}_n \leq \textbf{x},\, \alpha_n(Y_\des^{(1)}, \ldots, Y_\des^{(k_n)})^\top - \boldsymbol{\upalpha}_n \leq \textbf{y}\right),
\end{align*}
with $\textbf{x} = (x, \ldots, x)^\top, \textbf{y} = (y, \ldots, y)^\top, \boldsymbol{\upalpha}_n = (\alpha_n, \ldots, \alpha_n)^\top \in \R^{k_n}$. An application of Proposition~\ref{thm3.1} then yields, as $\nto$,
\begin{align*}
   &\big| \P\big(\alpha_n (\hat{M}_{n} -\alpha_n) \le x, \alpha_n (M_{n,\des} -\alpha_n) \le y  \big) -	P_n(x,y) \big| \\
    &= O\left(n^{-1/6}\log(k_n)^{7/6}\right) = o(1)\,.
\end{align*}
Finally, by the standard result \cite[Theorem 3]{sibuya1960bivariate} for maxima of bivariate Gaussian random vectors with correlation strictly less than $1$, we have that 
$$P_n(x,y) \overset{n\rightarrow\infty}{\longrightarrow} \Lambda(x) \Lambda(y),$$
since $\rho_n\to 0$ (see Corollary~\ref{cor2.7}), completing the proof of \eqref{eq:mainevlt2}.
\end{proof}

\begin{remark}
Due to the H\'{a}jek approximation error, the upper bound for the row-wise number of samples $k_n$ is stricter than that for the individual statistics $X_\inv$ and $X_\des$ given in \cite{dorr2022extreme}. In particular, this excludes the uniform triangular array ($k_n=n$). On the other hand, this new EVLT can be transferred to other individual and joint permutation statistics, and if H\'{a}jek's projection is not involved, we can achieve a subexponential bound for $k_n$ similarly as in \cite{dorr2022extreme}. Besides descents, some other examples of $m$-dependent permutation statistics on $S_n$ are \textit{peaks} (all indices $i$ with $\pi(i-1)<\pi(i)>\pi(i+1)$) and \textit{valleys} (all $i$ with $\pi(i-1) > \pi(i) < \pi(i+1)$). Since the proof of Theorem~\ref{thm4.3} does not rely on any special property of descents other than $m$-dependence, any other $m$-dependent permutation statistic could be combined with inversions. Furthermore, if we consider a permutation statistic consisting of one or two $m$-dependent components, then there is no need to use H\'{a}jek's projection and the corresponding part in the proof of Theorem~\ref{thm4.3} can be removed. 

To formalize these thoughts, let $\mathcal{W}$ be the system of classical Weyl groups or a subsystem (e.g., the family of symmetric groups) and let $(X_n)_{n \in \N}$ be a collection of random variables $X_n \negthickspace: W_n \rightarrow \N^d$ with $W_n \in \mathcal{W}$, $\rk(W) = n$, and with $d \in \{1,2\}$ fixed. Moreover, we assume that there is a representation 
\begin{equation}
    X_n = \sum_{i=1}^n f_i(Z_1, \ldots, Z_n) =: \sum_{i=1}^n X_n^{(i)} \label{4.12}
\end{equation}
for some independent sequence $Z_1, Z_2, \ldots$ of random variables and functions $f_n \negthickspace: \R^n \rightarrow \N^d,$ such that the following is satisfied:
\begin{itemize}
    \item If $d=1$, then we assume that $X_n$ is $m$-dependent (i.e., that all blocks $X_n^{(1)}, \ldots, X_n^{(n)}$ are $m$-dependent) for some $m \in \N$ chosen independently of $n$. Besides $X_\des$, an example of such a permutation statistic is the number of peaks or valleys.
    \item If $d=2$, then one component must be $m$-dependent. The other component must be $m$-dependent as well, or satisfy the condition of Theorem~\ref{thm2.2}, i.e., $\Var(X_n) \sim \Var(\hat{X}_n)$, where $\hat{X}_n$ is the H\'{a}jek projection of $X_n$ based on \eqref{4.12}. In the latter case, the connection to the H\'{a}jek projection needs to be established by ensuring proper bounds in \eqref{eq:mainevlt2} and \eqref{eq:dfdgs}. In light of \cite[Theorem 3]{sibuya1960bivariate}, it is required that the correlation between the two components is not $1$, but this is commonly trivial to verify. 
\end{itemize}
\end{remark}

\begin{theorem} \label{thm4.4.3}
Let $(W_n)_{n \in \N} \subseteq \mathcal{W}$ be a sequence of classical Weyl groups with $\rk(W_n) = n~\forall n \in \N$. Let $X_n$ be a permutation statistic as described above. Let $(X_{nj})_{j=1, \ldots, k_n}$ be a row-wise i.i.d.\ triangular array with $X_{n1} \overset{\D}{=} X_n,$ where:
\begin{enumerate}
    \item[\emph{(a)}] If $X_n$ is $m$-dependent, we assume $k_n = \exp\bigl(o(n^{1/7})\bigr).$
    \item[\emph{(b)}] If $X_n$ consists of two components, one of which is $m$-dependent and the other is not, but satisfies the condition of Theorem \emph{\ref{thm4.3},} then we assume $k_n\log(k_n) = o(n)$.
\end{enumerate}
Let $M_n := \max\{X_{n1}, \ldots, X_{nk_n}\}$ be the row-wise maximum. Let $\mu_n := \E(X_n),$ $s_n := \sigma(X_n),$ $a_n := s_n\alpha_{k_n},$ and $b_n := s_n\alpha_{k_n} + \mu_n$, which is understood component-wise in case of $d=2$. Then,
\[
    \forall \mathbf{x} \in \R^d \negmedspace: \PP(M_n \leq a_n \circ \mathbf{x} + b_n) \longrightarrow \Lambda_d(\mathbf{x})\,.
\]
\end{theorem}

\begin{proof}
In case of (b), the proof is identical to that of Theorem~\ref{thm4.3}. In case of (a), we only need to show \eqref{eq:mainevlt2}, while we replace $(\hat{M}_{n,\inv}, M_{n,\des})$ with the standardized maximum of $X_n$. We can apply Theorem~\ref{thm3.1} with $\mathfrak{p}(n) =k_n$ iterations of $X_n$. Therefore, we need to ensure that
\[
    n^{-1/6}\log(k_n)^{7/6} = o(1)\,,
\]
which exactly corresponds to the stated condition of $k_n = \exp\bigl(o(n^{1/7})\bigr)$. The claim follows the same way as in the proof of Theorem~\ref{thm4.3}.
\end{proof}


\noindent In conclusion, for independent and $m$-dependent permutation statistics, the high-dimensional Gaussian approximation allows to obtain a subexponential bound of $k_n,$ improving \cite[Theorem 5.1]{dorr2022extreme}. In particular, this applies to the following permutation statistics: 

\begin{corollary}
Let $Z_1, \ldots, Z_n \sim U(0,1)$ be i.i.d. The following three permutation statistics are in the max-domain of the Gumbel distribution, given a triangular array with row lengths satisfying $k_n = \exp\bigl(o(n^{1/7})\bigr) \negmedspace:$
\begin{itemize}
    \item the number of peaks $X_p := \displaystyle{\sum\nolimits_{i=2}^{n-1} \mathbf{1}\{Z_i > Z_{i-1}, Z_{i+1}\}}\hspace{0.3mm},$
    \item the number of valleys $X_v := \displaystyle{\sum\nolimits_{i=2}^{n-1} \mathbf{1}\{Z_{i-1}, Z_{i+1} > Z_i\}}\hspace{0.3mm},$
    \item the number of cycles $K_{0n},$ since according to \cite[Eq.\ 1.27]{arratia2003logarithmic}\emph{,} it has the independent sum decomposition $\displaystyle{K_{0n} = \sum\nolimits_{j=1}^n \Bin(1, j^{-1})}$.
\end{itemize}
\end{corollary}

\section{Signed and even-signed permutation groups} \label{section5}

The previous section established the EVLT for the joint distribution $(X_\inv, X_\des)^{\top}$~on the family of symmetric groups. We now work toward this result for $(X_\inv, X_\des)^{\top}$ on the \textit{signed and even-signed permutation groups} $B_n$ and $D_n$ that generalize the symmetric groups $S_n$. In this process, we also obtain the CLT for $(X_\inv, X_\des)^{\top}$ on these groups, which was already proven by \cite{fang2015rates} on symmetric groups.

The group $B_n$ of \textit{signed permutations} on $n$ letters arises from $S_n$ by assigning any combination of positive or negative signs to the entries of a permutation. The \textit{even-signed permutation group} $D_n$ is the subgroup of $B_n$ consisting of elements with an even number of negative signs. As in \cite[Section 2.1]{kahle2020counting}, we use the in-line notation 
\begin{align}
    B_n = \bigl\{\pi = (\pi(1), \ldots, \pi(n)) :~&\pi(i) \in \{\pm 1, \ldots, \pm n\}~\forall i, \nonumber \\
    & \{|\pi(1)|, \ldots, |\pi(n)|\} = \{1, \ldots, n\}\bigr\}, \label{eq5}
\end{align}
and we have $D_n = \{\pi \in B_n \negmedspace: \pi(1)  \pi(2)  \cdots  \pi(n) > 0\}$.

The combinatorial representation of inversions on $B_n$ and $D_n$ is given by
\begin{equation}
    X_\inv(\pi) = \begin{cases} |\Inv^+(\pi)| + |\Inv^-(\pi)| + |\Inv^\circ(\pi)|, & \text{ if } \pi \in B_n, \\ |\Inv^+(\pi)| + |\Inv^-(\pi)|\,, & \text{ if } \pi \in D_n\,, \end{cases} \label{5}
\end{equation}
for the disjoint sets
\begin{align*}
    \Inv^+(\pi) &= \{1 \leq i < j \leq n \mid \pi(i) > \pi(j)\}, \\
    \Inv^-(\pi) &= \{1 \leq i < j \leq n \mid -\pi(i) > \pi(j)\}, \\
    \Inv^\circ \hspace{0.85mm} (\pi) &= \{1 \leq i \leq n \mid \pi(i) < 0\}.
\end{align*}
When expanding the in-line notation in \eqref{eq5} by setting $\pi(0) := 0$, we can also represent $X_\des$ on $B_n$ and $D_n$ as
\begin{subequations}
\begin{align}
    X_\des^B(\pi) &= \sum_{i=0}^{n-1} \textbf{1}\{\pi(i) > \pi(i+1)\}, \label{14a} \\
    X_\des^D(\pi) &= \textbf{1}\{-\pi(2) > \pi(1)\} + \sum_{i=1}^{n-1} \textbf{1}\{\pi(i) > \pi(i+1)\}. \label{14b}
\end{align}
\end{subequations}
For more details and formal proofs, see \cite[Sections 8.1 and 8.2]{bjorner2006combinatorics}.

To draw elements uniformly from $B_n$, one can first draw some uniform $\pi \in S_n$ and then multiply each $\pi(i)$ with a Rademacher variable independent of everything else. Instead of the Rademacher variables, we propose a more general approach by drawing signs with a fixed $p$-bias, i.e., each sign is $-1$ with probability $p \in [0, 1]$ and $+1$ with probability $q := 1-p$. This yields a family of probability measures on $B_n$, where the case $p=1/2$ corresponds to the uniform distribution on $B_n$, while if $p=0$, all mass is on the symmetric group $S_n \subseteq B_n$. In this sense, we have a continuous transition from $S_n$ to the groups $B_n$ and $D_n$. 

A corresponding probability distribution on the even-signed permutation group $D_n$ is obtained by first choosing the unsigned permutation $|\pi| \in S_n$ uniformly and then assigning $n-1$ signs for the entries $\pi(1), \ldots, \pi(n-1)$ with $p$-bias, and finally specifying the sign of $\pi(n)$ such that $\pi(1) \cdots \pi(n-1)\pi(n)>0$.

\begin{definition}
Let $p \in [0, 1]$ and $q := 1-p$. Then, the $p$-\textit{biased signed permutations measure} on the group $B_n$ is the probability measure induced by the point masses
\[
    \PP(\{\pi\}) = \frac{1}{n!}p^{\text{neg}(\pi)}q^{n-\text{neg}(\pi)},\qquad \pi \in B_n\,,
\]
where $\text{neg}(\pi)$ denotes the number of negative signs in $\pi$ and we use the convention $0^0 := 1$. The $p$-biased signed permutations measure on $D_n$ is derived as described above.
\end{definition}

Therefore, the entries of $\pi$, with $\pi$ distributed according to the $p$-biased signed permutations measure, can be represented by i.i.d.\ random variables $Z_1, \ldots, Z_n$ with $\forall i = 1, \ldots, n \negmedspace: Z_i = U_iR_i,$ where $R_i$ is a $\pm 1$-valued random variable with
\[
    \PP(R_i = -1) = p \quad \text{ and } \quad \PP(R_i = 1) = q,
\]
and $U_i \sim U(0,1)$ is independent of $R_i$. The probability distribution function of $Z_1$ is
\[
    F_p(z) := \P(Z_1\le z) = \begin{cases} pz + p, & \quad z\in [-1,0], \\ qz + p, &  \quad z\in [0,1], \end{cases}
\]
and we simply write $Z_1 \sim \GR(p)$ (generalized Rademacher with parameter $p$). Note the special cases $\GR(0) = U(0,1),$ $\GR(1/2) = U(-1,1)$ and $\GR(1) = U(-1,0)$. Figure~\ref{fig4} illustrates $F_p$ in the cases of $p=0, p=1/4$ and $p=3/4$. Accordingly, the Lebesgue density of $\GR(p)$ is $f_p(z) = p\textbf{1}\{-1 < z < 0\} + q\textbf{1}\{0 < z < 1\}$.

\begin{figure}[ht]
    \centering
    \includegraphics[scale=1.15]{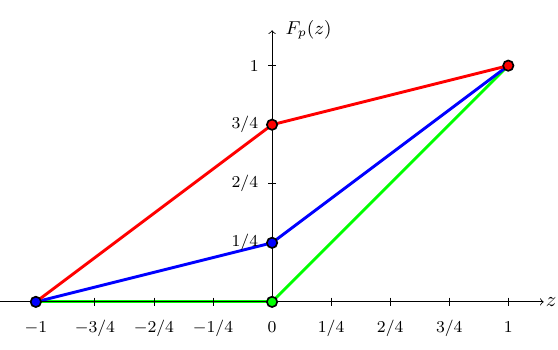}
    \caption{Probability distribution functions of $\GR(p)$ for $p=0$ (green), $p=1/4$ (blue) and $p=3/4$ (red).}
    \label{fig4}
\end{figure}

\begin{remark}
Let $X_\inv^B$ denote the random number of inversions on $B_n$ and let $X_\inv^D$ denote that on $D_n$. According to \eqref{5}, we can write
\begin{subequations}
\begin{align}
    X_\inv^B &= \sum_{i<j} \textbf{1}\{Z_i > Z_j\} + \sum_{i<j} \textbf{1}\{-Z_i > Z_j\} + \sum_{i=1}^n \textbf{1}\{Z_i < 0\}, \label{6} \\
    X_\inv^D &= \sum_{i<j} \textbf{1}\{Z_i > Z_j\} + \sum_{i<j} \textbf{1}\{-Z_i > Z_j\}. \label{7}
\end{align}
\end{subequations}
Furthermore, \eqref{14a} and \eqref{14b} translate to
\begin{subequations}
\begin{align}
    X_\des^B &= \sum_{i=1}^{n-1} \textbf{1}\{Z_i > Z_{i+1}\} + \textbf{1}\{Z_1 < 0\}, \label{13a} \\
    X_\des^D &= \sum_{i=1}^{n-1} \textbf{1}\{Z_i > Z_{i+1}\} + \textbf{1}\{-Z_2 > Z_1\}. \label{13b}
\end{align}
\end{subequations}
In what follows, we use $X_\inv$ and $X_\des$ as an umbrella notation for the numbers of inversions and descents on each of the groups $S_n, B_n$ and $D_n$. Again, both \eqref{13a} and \eqref{13b} give a $1$-dependent representation of $X_\des$, which yields a $1$-dependent representation of $(\hat{X}_\inv, X_\des)^\top$. 

In the uniform case ($p=1/2$), the means and variances are known by \cite[Corollaries 3.2 and 4.2]{kahle2020counting}. We now  calculate them within the general $p$-bias framework. To this end, we first observe that for any $i<j$,
\begin{align*}
\P(-Z_i > Z_j) &=\E ( \P(-Z_i > Z_j| Z_i)) = \E(F_p(-Z_i))= \E(\E(F_p(-Z_i)|U_i))\\
&= \E(p F_p(U_i)+q F_p(-U_i)) =\E(p(qU_i+p)+q(-pU_i+p)) = p\,.
\end{align*}
Then, it follows straightforwardly from \eqref{6} and \eqref{7} that 
\begin{align*}
    \E\left(X_\inv^B\right) &= \binom{n}{2}\left(p + \frac{1}{2}\right) + np, & \E\left(X_\inv^D\right) &= \binom{n}{2}\left(p + \frac{1}{2}\right).
\end{align*}
The formula for $\Var(X_\inv)$ is given in the next lemma.
\end{remark}

\begin{lemma}[see Subsection \ref{sec9.4} for the proof] \label{lemma5.4}
    On the $p$-biased (even-)signed permutation groups, we have
    \begin{align*}
        \Var\left(X_\inv^B\right) &= \left(-\frac{1}{3}p^2 + \frac{1}{3}p + \frac{1}{36}\right)n^3 - \left(3p^3 - 4p^2 + p - \frac{1}{24}\right)n^2 \\
        &\hspace{2mm} \qquad +~\left(3p^3 - \frac{14}{3}p^2 + \frac{5}{3}p - \frac{5}{72}\right)n, \\
        \Var\left(X_\inv^D\right) &= \left(-\frac{1}{3}p^2 + \frac{1}{3}p + \frac{1}{36}\right)n^3 - \left(p^3 - 2p^2 + p - \frac{1}{24}\right)n^2 \\
        &\hspace{2mm} \qquad +~\left(p^3 - \frac{5}{3}p^2 + \frac{2}{3}p - \frac{5}{72}\right)n.
    \end{align*}
    In particular, if $p=0$ or $p=1/2,$ we obtain the results in \emph{\cite[Corollaries 3.2 and 4.2]{kahle2020counting}}.
\end{lemma}

For the variance of $\hat{X}_\inv$ on $B_n$ and $D_n$, we get the same leading term as in Lemma~\ref{lemma5.4}, regardless of $p$. Hence, we obtain the H\'{a}jek approximation statement from Lemma~\ref{lemma2.4} on the groups $B_n$ and $D_n$ with $p$-bias. 

\begin{lemma}[see Subsection \ref{sec9.5} for the proof] \label{lemma5.5}
    On the $p$-biased (even-)signed permutation groups, we also have
    \[\Var(\hat{X}_\inv) = \left(-\frac{1}{3}p^2 + \frac{1}{3}p + \frac{1}{36}\right)n^3 + O(n^2), \]
    so $\hat{Y}_\inv = Y_\inv + o_{\mathbb{P}}(1)$ applies according to Theorem~\ref{thm2.2}.  
\end{lemma}

The leading term, as a function of $p$, has no zeros in $[0,1]$ and assumes its global maximum at $p=1/2,$ which is the unbiased case. This means that the order of $\Var(X_\inv)$ and $\Var(\hat{X}_\inv)$ is guaranteed to be cubic in $n$. 

From Lemma~\ref{lemma5.5}, we obtain an extension of Corollary~\ref{cor2.5}, which we present as a general statement on all three families of classical Weyl groups.

\begin{corollary} \label{cor5.6}
    Let $W$ be a classical Weyl group of rank $n$, that is, $W\in \{S_n,B_n,D_n\}$. Set
    $Z_0 := -\infty$ if $W = S_n$, $Z_0:= 0$ if  $W = B_n$  and $Z_0:= -Z_2$ if $W=D_n$. Then,
    {\small\begin{align*}
        \begin{pmatrix} \!\hat{X}_\inv \! \\ \!X_\des \! \end{pmatrix} &\!=\! \begin{pmatrix} \E(X_\inv \mid Z_1) \\ \mathbf{1}\{Z_0 > Z_1\} \end{pmatrix} + \ldots + \begin{pmatrix} \E(X_\inv \mid Z_{n-1}) \\ \mathbf{1}\{Z_{n-2} > Z_{n-1}\} \end{pmatrix} + \begin{pmatrix} \E(X_\inv \mid Z_n) \!-\! (n\!-\! 1)\E(X_\inv) \\ \mathbf{1}\{Z_{n-1}>Z_n\} \end{pmatrix}
    \end{align*}}
    \hspace{-2.7mm} is a $1$-dependent decomposition of $(\hat{X}_\inv, X_\des)^\top$. On $B_n$ and $D_n$, this applies with any sign bias.
\end{corollary}

\begin{lemma}[see Subsections~\ref{sec9.6},~\ref{sec9.7} for the proof] \label{lemma5.6}
On both of the groups $B_n$ and $D_n$ with $p$-bias, it holds that 
\begin{enumerate}
\item[\emph{a)}] $\rho(X_\inv, X_\des) \longrightarrow 0$ as $\nto$, and 
\begin{align*}
    \Cov(X_\inv^B, X_\des^B) &= (n-1)\left(\frac{p^2}{2} + p^2q - \frac{p}{2} + \frac{1}{4}\right) + (p - p^2), \\
    \Cov(X_\inv^D, X_\des^D) &= (n-1)\left(\frac{p^2}{2} + p^2q - \frac{p}{2} + \frac{1}{4}\right) + p^2.
\end{align*}  
\item[\emph{b)}] $\Cov(\hat{X}_\inv, X_\des) = \Theta(1/n),$ so again, $\rho(\hat{X}_\inv, X_\des) \longrightarrow 0$ as $\nto$.
\end{enumerate}
\end{lemma}

With all this information at hand, we can use Proposition~\ref{thm3.1} to obtain the CLT of $(X_\inv, X_\des)^\top$ on $B_n$ and $D_n$ (while we could have done so for $S_n$, this is a novel contribution only for $B_n$ and $D_n$). Moreover, all arguments in the proof of the extreme value limit Theorem~\ref{thm4.3} apply on $B_n$ and $D_n$.

\begin{theorem} \label{thm4.7} For the joint statistic $(X_\inv, X_\des)^\top$ on signed or even-signed permutation groups with $p$-bias, the following hold. 
\begin{enumerate}
    \item[a)] $(X_\inv, X_\des)_{n\ge 1}^\top$ satisfies the CLT, i.e.,
    \[
        (Y_\inv, Y_\des)^\top = \left(\frac{X_\inv - \E(X_\inv)}{\sqrt{\Var(X_\inv)}}, \frac{X_\des - \E(X_\des)}{\sqrt{\Var(X_\des)}}\right)^\top \overset{\D}{\longrightarrow} \NN_2(0, \mathrm{I}_2)\,, 
    \]
    as $n\to\infty.$
    \item[b)] The statement of Theorem~\ref{thm4.3} holds if $W_n$ is an arbitrary sequence of classical Weyl groups with $\rk(W_n) = n$ for all $ n \in \N$ and with $k_n$ chosen so that $k_n\log(k_n) = o(n)$.
\end{enumerate}
\end{theorem}

\begin{proof}
a): Due to Corollary~\ref{cor2.5} and Slutsky's Theorem, it suffices to show that $(\hat{Y}_\inv, Y_\des)^\top \overset{\D}{\longrightarrow} \NN_2(0, \mathrm{I}_2)$. By \eqref{eq:1}, we can find $c_i^{(n)}, d_i^{(n)},$ $i=1, \ldots, n$, such that
\begin{align}
\begin{pmatrix} \hat{X}_\inv -\E(\hat{X}_\inv)  \\ X_\des-\E(X_\des) \end{pmatrix} &= 
\sum_{k=2}^n \begin{pmatrix} c_k^{(n)}Z_k - d_k^{(n)} \\ \mathbf{1}\{Z_{k-1} > Z_k\} - 1/2 \end{pmatrix} + \begin{pmatrix} c_1^{(n)}Z_1 - d_1^{(n)} \\ \textbf{1}\{Z_0 > Z_1\} - \PP(Z_0 > Z_1) \end{pmatrix}
\end{align}
is a sum of $1$-dependent random vectors with mean zero. Setting 
\begin{align*} 
    X_1^{(n)}&:= \begin{pmatrix}  \left(c_1^{(n)}Z_k - d_1^{(n)}\right)/\sqrt{\Var(\hat{X}_\inv)} \\ \left(\textbf{1}\{Z_0 > Z_1\} - \PP(Z_0 > Z_1)\right)/\sqrt{\Var(X_\des)} \end{pmatrix}, \\
    X_k^{(n)}&:= \begin{pmatrix} \left(c_k^{(n)}Z_k - d_k^{(n)}\right)/\sqrt{\Var(\hat{X}_\inv)} \\ (\mathbf{1}\{Z_k > Z_{k+1}\} -1/2)/\sqrt{\Var(X_\des)} \end{pmatrix}, \qquad k=2,\ldots,n\,,
\end{align*}
we obtain the representation $(\hat{Y}_\inv, Y_\des)^\top = \sum_{k=1}^n X_k^{(n)}=:X^{(n)}$. The covariance matrix of $X^{(n)}$ (see \eqref{eq4}) is given by $\Sigma^{(n)}=\begin{pmatrix} 1 & \rho_n \\ \rho_n & 1 \end{pmatrix}$, where $\rho_n:=\rho(\hat{X}_\inv, X_\des)$. An application of Corollary~\ref{cor3.1} yields that for $\mathcal{N}_n\sim \NN(0,\Sigma^{(n)})$, 
\begin{align*}
    \sup_{u \in \R^2} |\PP(X^{(n)} \leq u) - \PP(\mathcal{N}_n \leq u)| &\le r_n(\A^{\text{re}})=
    O\left(n^{-1/6} \log(n)^{7/6}\right)\,.
\end{align*}
In combination with the fact that the correlation $\rho_n$ vanishes in the limit (see Corollary~\ref{cor2.7}), we can conclude that $(\hat{Y}_\inv, Y_\des)^\top \overset{\D}{\longrightarrow} \NN_2(0, \mathrm{I}_2)$. The proof of b) is the same as the proof of Theorem~\ref{thm4.3}.
\end{proof}


At last, we consider direct products of classical Weyl groups. Let $W = \prod_{i=1}^l W_i$ be such a product, where each $W_i$ is one of $S_n, B_n$ or $D_n$, and $l$ is a fixed positive integer. By \cite[Lemma 2.2]{kahle2020counting}, we know that 
\[
    X_\inv^W = \sum_{i=1}^l X_\inv^{W_i}
\]
is a sum of independent random variables, implying $\Var(X_\inv^W) = \sum_{i=1}^l \Var(X_\inv^{W_i})$. Let $X_\inv^{W_i}$ be constructed from variables $Z_1^{(i)}, \ldots, Z_{n_i}^{(i)}$, where $n_i$ denotes the number of letters on which the group $W_i$ acts, and each $Z_j^{(i)}$ is $\GR(p_i)$ for some $p_i \in [0,1]$, and the entire collection of all  $Z_j^{(i)}$ is independent. Setting $n := n_1 + \ldots + n_l$, the overall \Hajek\ projection $\hat{X}_\inv^W$ of $X_\inv^W$ is
\[
    \hat{X}_\inv^W = \sum_{i=1}^l \sum_{j=1}^{n_i} \E\left(X_\inv^W \mid Z_j^{(i)}\right) - (n-1)\E(X_\inv^W)\,,
\]
where 
$\E\left(X_\inv^W \mid Z_j^{(i)}\right) = \sum_{k=1}^l \E\left(X_\inv^{W_k} \mid Z_j^{(i)}\right)$.
If $k \not= i,$ then $X_\inv^{W_k}$ is independent of $Z_j^{(i)},$ which means that in this case $\E\left(X_\inv^{W_k} \mid Z_j^{(i)}\right) = \E(X_\inv^{W_k})$ is constant. We therefore obtain 
\[
    \Var(\hat{X}_\inv^W) = \sum_{i=1}^l \sum_{j=1}^{n_i} \Var\left(\E\left(X_\inv^{W_i} \mid Z_j^{(i)}\right)\right) = \sum_{i=1}^l \Var(\hat{X}_\inv^{W_i}).
\]
For any $W_i,$ we have $\Var(X_\inv^{W_i}) \sim \Var(\hat{X}_\inv^{W_i})$. Furthermore, all variances are cubic as seen in Lemmas~\ref{lemma2.4},~\ref{lemma5.4},~\ref{lemma5.5} and \cite[Corollary 3.2]{kahle2020counting}, i.e., we have 
\[
    \Var(X_\inv^{W_i}) = c_i n_i^3 + O(n_i^2)=\Var(\hat{X}_\inv^{W_i}), \qquad n_i \to \infty,
\]
where $c_i := -\frac{1}{3}p_i^2 + \frac{1}{3}p_i + \frac{1}{36}$.  It is seen from the calculations in the proofs of Lemmas~\ref{lemma2.4} and~\ref{lemma5.5} that $\Var(X_\inv) - \Var(\hat{X}_\inv) = \Theta(n^2)$. So we can write 
\begin{align}
    \Var(X_\inv^W) &= \sum_{i=1}^l \Big(c_in_i^3 + \alpha_in_i^2 + O(n_i)\Big), & \Var(\hat{X}_\inv^W) &= \sum_{i=1}^l \Big(c_in_i^3 + \beta_in_i^2 + O(n_i)\Big), \label{18}
\end{align}
with $\alpha_i \not= \beta_i$ for all $i$. 

Consider a sequence $(W_n)_{n \in \N}$ of products as introduced above, assuming that the number $l$ of components remains bounded. Then, we see that $\Var(X_\inv^W) \sim \Var(\hat{X}_\inv^W)$ holds as $n\to \infty$, since the cubic terms are equal and cannot be dominated by the quadratic terms. Thus, to obtain the CLT, the H\'{a}jek projection approximation is sufficient and the above considerations show the following extension of Theorem~\ref{thm4.7}. 

\begin{corollary}
On bounded products of classical Weyl groups, it holds that $Y_\inv = \hat{Y}_\inv + o_{\PP}(1)$, and  $(X_\inv, X_\des)^\top$ satisfies the CLT.
\end{corollary}  

To repeat the proof of the extreme value limit Theorem~\ref{thm4.3}, there is another issue to consider, namely the bounds \eqref{p1} and \eqref{p2}, which require a suitable control of
\begin{equation}
    1 - \frac{\Var(\hat{X}_\inv)}{\Var(X_\inv)}. \label{19}
\end{equation}
Since the number of components of $W_n$ is bounded, we can assume w.l.o.g. that the components are sorted decreasingly by rank, meaning $n_1$ is the largest rank with $n_1 = \Theta(n)$. We can also assume that each group has exactly $l$ components (as groups with fewer components can be filled with components of $S_1$, not giving any further inversions and descents).

\begin{theorem}
For fixed $l\in \N$, let $W_n = \prod_{i=1}^l W_{n,i}$ be products of finite Coxeter groups with $\rk(W_n) = n$ $\forall n \in \N$. Let $k_n\log(k_n) = o(n)$ and let $(X_\inv^{(j)}, X_\des^{(j)})^\top,$ $j=1, \ldots, k_n$ be independent copies of $(X_\inv, X_\des)^\top$ on $W_n$. Let $M_{n,\inv}, M_{n,\des}$ be defined as in \eqref{eq:minv}. Then, the statement of Theorem \emph{\ref{thm4.3}} applies for $(W_n)_{n \in \N}$, that is, we obtain \eqref{eq:mainevlt} again.
\end{theorem}

\begin{proof}
The proof of Theorem~\ref{thm4.3} carries over almost seamlessly, we only need to check the bound of \eqref{19}. We can rephrase \eqref{18} as
\begin{align*}
    \Var(\hat{X}_\inv^{W_n}) &= \sum_{i=1}^l c_in_i^3 + \alpha n^2 + O(n), & \Var(X_\inv^{W_n}) &= \sum_{i=1}^l c_in_i^3 + \beta n^2 + O(n).
\end{align*}
Then, 
\[
    1 - \frac{\Var(\hat{X}_\inv^{W_n})}{\Var(X_\inv^{W_n})} = 1 - \frac{\sum_{i=1}^l c_in_i^3 + \alpha n^2 + O(n)}{\sum_{i=1}^l c_in_i^3 + \beta n^2 + O(n)}.
\]
Depending on whether the residual $\beta n^2 + O(n)$ is positive or negative, we can bound this in both directions (assuming it is positive) via
\begin{align*}
    \eqref{19} &\geq 1 - \frac{\sum_{i=1}^l c_in_i^3 + \alpha n^2 + O(n)}{\sum_{i=1}^l c_in_i^3} = \frac{\alpha n^2 + O(n)}{\sum_{i=1}^l c_in_i^3} = O\left(\frac{1}{n}\right), \\
    \eqref{19} &= \frac{(\beta - \alpha)n^2 + O(n)}{\sum_{i=1}^l c_in_i^3 + \beta n^2 + O(n)} \leq \frac{(\beta - \alpha)n^2 + O(n)}{\sum_{i=1}^l c_in_i^3} = O\left(\frac{1}{n}\right).
\end{align*}
Therefore, we have the same bound for \eqref{19} as in the proof of Theorem~\ref{thm4.3}. 
\end{proof}
 
\section{Conclusion and Outlook} \label{section8}

In this work, we proved both a CLT and an EVLT for the joint statistic $(X_\inv, X_\des)$ on all classical Weyl groups, as well as their bounded products. This addresses one of the open questions raised in \cite{dorr2022extreme} and gives a significant extension to \cite{fang2015rates}, which only covered the CLT on symmetric groups. We benefited from the fact that the number of inversions $X_\inv$ can be suitably approximated by its H\'{a}jek projection, enabling us to apply Gaussian approximation theory for $m$-dependent vectors. In comparison with the univariate results in \cite{dorr2022extreme}, the triangular array could not be stretched as generously, because the involvement of H\'{a}jek's projection required stronger assumptions. 

On the symmetric groups $S_n$, both common inversions and descents are special instances of so-called \textit{generalized inversions} or \textit{$d$-inversions}. For any choice of $d \in \{1, \ldots, n-1\}$, the statistic of $d$-inversions counts only inversions over pairs $(i,j)$ with $1 \leq j-i \leq d$. It is interesting to choose $d = d_n$ as a function of $n$. The asymptotic normality of $d$-inversions under suitable conditions for $d_n$ has been shown in \cite{meier2022central}. Accordingly, it is worthwhile to investigate the constraints on $d_n$ that guarantee the EVLT in the way of Theorem~\ref{thm4.3}.

Some further open questions remain as well, especially about the \textit{two-sided Eulerian statistic} $X_T(\pi) := X_\des(\pi) + X_\des(\pi^{-1})$. This statistic is known to satisfy a CLT according to \cite{bruck2019central} by the method of dependency graphs, but its extreme value behavior still remains unknown. 
Likewise, the extreme asymptotics of the joint distribution $\bigl(X_\des(\pi), X_\des(\pi^{-1})\bigr)$ give another open question.


\section{Appendix: remaining proofs} \label{section9}

\subsection{Proof of Lemma~\ref{lemma2.6}a)} \label{sec9.1}

We compute $\Cov(X_\inv, X_\des)$ from \eqref{1.1} and \eqref{1.2}, obtaining that it has lower order than $\sqrt{\Var(X_\inv)\Var(X_\des)}$. According to \eqref{1.1} and \eqref{1.2}, we have 
\begin{align*}
    \Cov(X_\inv, X_\des) &= \sum_{1\le i<j\le n} \sum_{k=1}^{n-1} \Cov(\textbf{1}\{Z_i > Z_j\}, \textbf{1}\{Z_k > Z_{k+1}\})\\
    &=\sum_{1\le i<j\le n} \sum_{\substack{k\in\{i-1,i,j-1,j\}\\ 1\le k \le n-1}} \Cov(\textbf{1}\{Z_i > Z_j\}, \textbf{1}\{Z_k > Z_{k+1}\})\,,
\end{align*}
where we used that if $k \notin \{i-1, i, j-1, j\}$, then the events $\{Z_i > Z_j\}$ and $\{Z_k > Z_{k+1}\}$ are independent and therefore $\Cov(\textbf{1}\{Z_i > Z_j\}, \textbf{1}\{Z_k > Z_{k+1}\}) = 0$. In what follows, we analyze the case $k \in \{i-1, i, j-1, j\},$ first assuming that all these numbers are distinct. Additionally, we temporarily ignore the boundary cases  $i=1$ (where $k=i-1$ is outside the range) or $j=n$ (where $k=n$ is within the range but the variable $Z_{k+1}$ compared with $Z_k$ is not). This gives four possible constellations: 
\begin{itemize}
    \item type A: $k+1=i$ and $j>k+2$,
    \item type B: $k=i$ and $j > k+1$,
    \item type C: $k+1=j$ and $i < k$,
    \item type D: $k=j$ and $i < k-1$.
\end{itemize}
For type A, we have 
\begin{align*}
    \Cov(\textbf{1}\{Z_i > Z_j\}, \textbf{1}\{Z_{i-1} > Z_i\}) &= \PP(Z_i > Z_j, Z_{i-1} > Z_i) - \PP(Z_{i} > Z_j)\PP(Z_{i-1} >  Z_{i})\\
    &= \P(Z_{i-1}>Z_i>Z_j)-\tfrac{1}{4}\\
    &= \tfrac{1}{6} - \tfrac{1}{4} = -\tfrac{1}{12},
\end{align*}
since each of the six possible orderings of $Z_{i-1},Z_i,Z_j$ is equally likely as they are independent $U(0,1)$ variables.
For type B, we get
\begin{align*}
\Cov(\textbf{1}\{Z_i > Z_j\}, \textbf{1}\{Z_i > Z_{i+1}\}) &= \P(Z_i=\max\{Z_i, Z_{i+1} ,Z_j\}) - \tfrac{1}{4}\\
&= \tfrac{1}{3} - \tfrac{1}{4} = \tfrac{1}{12}
\end{align*}
since each of $Z_i, Z_{i+1} ,Z_j$ are equally likely to be the maximum.
Types C and D are handled the same way. For type C, we have
$\Cov(\textbf{1}\{Z_i > Z_j\}, \textbf{1}\{Z_{j-1} > Z_j\}) = \tfrac{1}{12},$
and for type D,
$\Cov(\textbf{1}\{Z_i > Z_j\}, \textbf{1}\{Z_{j} > Z_{j+1}\}) = -\tfrac{1}{12}$. So if $1<i<i+1<j<n$, the inner sum 
\[\sum_{k\in\{i-1,i,j-1,j\}} \Cov(\textbf{1}\{Z_i > Z_j\}, \textbf{1}\{Z_k > Z_{k+1}\})\]
consists of two canceling pairs of $1/12$ and $-1/12,$ and vanishes altogether. Figure~\ref{fig11} displays the passage of $k$ over the indices $1, \ldots, n$ and the positions of the positive and negative covariances, in which we see that the positive signs are located inside, while the negative ones are located outside. With the help of this figure, we can also see what happens when $1<i<i+1<j<n$ does not hold: 

\begin{figure}[t]
    \centering
    \includegraphics[scale=1.2]{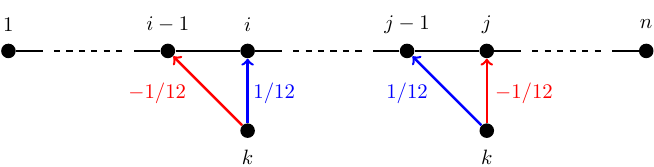}
    \caption{Canceling pairs of positive and negative covariances as $k$ passes through $1, \ldots, n-1$ in non-exceptional situations. The covariance is zero for all other values of $k$.}
    \label{fig11}
\end{figure}

\begin{itemize}
    \item If $i$ and $j$ are subsequent, i.e., $j=i+1,$ then the two positive contributions in Figure~\ref{fig11} collide. If additionally $k=i$, we obtain $\Cov(\textbf{1}\{Z_i > Z_j\}, \textbf{1}\{Z_k > Z_{k+1}\}) = \Cov(\textbf{1}\{Z_i > Z_{i+1}\}, \textbf{1}\{Z_i > Z_{i+1}\}) = \Var(\textbf{1}\{Z_i > Z_{i+1}\}) = 1/4$.
    \item If $i=1,$ then the leftmost negative contribution in Figure~\ref{fig11} disappears.
    \item If $j=n,$ then the rightmost negative contribution in Figure~\ref{fig11} disappears.
\end{itemize}
As these situations are not mutually disjoint, we obtain the following list of exceptional cases and their contributions $C_{ij} := \sum\nolimits_{k=1}^{n-1} \Cov(\textbf{1}\{Z_i > Z_j\}, \textbf{1}\{Z_k > Z_{k+1}\})$:
\begin{itemize}
    \item (E1): $j=i+1,$ but neither $i=1$ nor $j=n \Longrightarrow C_{ij} = 1/4 - 1/6 = 1/12$, 
    \item (E2): $i=1$ and $j=3, \ldots, n-1 \Longrightarrow C_{ij} = 1/12$,
    \item (E3): $j=n$ and $i=2, \ldots, n-2 \Longrightarrow C_{ij} = 1/12$,
    \item (E4): $i=n-1$, $j=n \Longrightarrow C_{ij} = 1/4 - 1/12 = 1/6$, 
    \item (E5): $i=n-1$, $j=n \Longrightarrow C_{ij} = 1/6$,
    \item (E6): $i=1$, $j=n \Longrightarrow C_{ij} = 1/6$.
\end{itemize}
As an example, we display situation (E1) in Figure~\ref{fig9}.
\begin{figure}[h]
    \centering
    \includegraphics[scale=1.25]{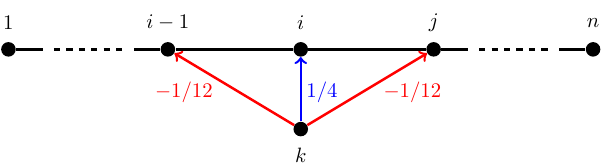}
    \caption{Display of covariances for the exceptional case (E1) when $i$ and $j=i+1$ are subsequent.}
    \label{fig9}
\end{figure} \par \noindent
Taking the contributions and frequencies of (E1), $\ldots,$ (E6) into account, we obtain the exact result
\[\Cov(X_\inv, X_\des) = \underbrace{(n-3)\left(\frac{1}{4}-\frac{1}{6}\right)}_{(\mathrm{E}1)} + \underbrace{2(n-3)\frac{1}{12}}_{(\mathrm{E}2,~\mathrm{E}3)} + \underbrace{2\left(\frac{1}{4} - \frac{1}{12}\right)}_{(\mathrm{E}4,~\mathrm{E}5)} + \underbrace{\frac{1}{6}}_{(\mathrm{E}6)} = \frac{n-1}{4}.\]
The claim follows. \qed

\subsection{Proof of Lemma~\ref{lemma2.6}b)} \label{sec9.2}

Recall that $Z_1,\ldots, Z_n$ are i.i.d.\ $U(0,1)$. By Lemma~\ref{lemma2.4}, we have 
\[\hat{X}_\inv = \sum_{j=1}^n (n-2j+1)Z_j + \frac{1}{2} \binom{n}{2},\]
which, together with the definition of $X_\des$, yields
\begin{align*}
    \Cov(\hat{X}_\inv, X_\des) &= \Cov\left(\sum_{j=1}^n (n-2j+1)Z_j, \sum_{k=1}^{n-1} \textbf{1}\{Z_k > Z_{k+1}\}\right) \\
    &= \Big( \sum_{j=2}^{n-1}+ \sum_{j\in \{1,n\}}\Big) \sum_{k=1}^{n-1} (n-2j+1)\Cov(Z_j, \textbf{1}\{Z_k > Z_{k+1}\})\\
    & =: T_1+T_2\,.
\end{align*}
In view of the independence of $Z_1, \ldots, Z_n$, we get
$$T_1= \sum_{j=2}^{n-1} (n-2j+1) \big(\Cov(Z_j, \textbf{1}\{Z_j < Z_{j-1}\}) + \Cov(Z_j, \textbf{1}\{Z_j > Z_{j+1}\})\big)=0\,,$$
where for the last equality, we used
\begin{align*}
    &\Cov(Z_j, \textbf{1}\{Z_j < Z_{j-1}\}) + \Cov(Z_j, \textbf{1}\{Z_j > Z_{j+1}\})\\
    & = \E(Z_j\textbf{1}\{Z_j < Z_{j-1}\}) + \E(Z_j\textbf{1}\{Z_j > Z_{j+1}\}) - \frac{1}{2}\\
     & = \E(Z_j\textbf{1}\{Z_j < Z_{j-1}\}) + \E(Z_j\textbf{1}\{Z_j > Z_{j-1}\}) - \frac{1}{2}\\
     &=\E(Z_j) -\frac{1}{2} =0\,.
\end{align*}
As $Z_1\textbf{1}\{Z_1 < Z_{2}\}$ is a function in two uniform variables with joint density $f \negthickspace: \R^2 \rightarrow \R, (x,y) \mapsto \textbf{1}\{(x,y) \in [0,1]^2\}$, we can apply Fubini's Theorem to obtain
\begin{align*}
    \E(Z_1\textbf{1}\{Z_1 < Z_{2}\}) &= \int_{[0,1]^2} x\textbf{1}\{x < y\}\,\text{d}(x,y) \\
    &= \int_0^1 x\left(\int_0^1 \textbf{1}\{x<y\}\dy\right)\dx \\
    &= \int_0^1 x(1-x)\,\dx = \frac{1}{6}.
\end{align*}
Therefore, we get for $T_2$: 
\begin{align*}
    T_2&= (n-1)\Cov(Z_1, \textbf{1}\{Z_1 > Z_{2}\})
   - (n-1)\Cov(Z_n, \textbf{1}\{Z_{n-1} > Z_{n}\})\\
   &= (n-1) \big( \E(Z_1 \textbf{1}\{Z_1 > Z_{2}\})- \E(Z_1 \textbf{1}\{Z_1 < Z_{2}\}) \big) \\
   &= (n-1) \big( 1/2 - 2 \E(Z_1 \textbf{1}\{Z_1 < Z_{2}\}) \big)
   =\frac{n-1}{6}\,, 
\end{align*}
which shows that $\Cov(\hat{X}_\inv, X_\des) = \displaystyle{\frac{n-1}{6}}$. \qed

\subsection{Proof of Lemma~\ref{lemma5.4}} \label{sec9.4}

We follow the instructive calculation of $\Var(X_\inv)$ in the uniform case provided in \cite[Section 3]{kahle2020counting}. So, the main task is to calculate $\E(X_\inv^2)$. For $X_\inv^B,$ we recall \eqref{6} and use the abbreviations
\[X_\inv^B = \underbrace{\sum_{i<j} \textbf{1}\{Z_i > Z_j\}}_{=:~X^+} + \underbrace{\sum_{i<j} \textbf{1}\{-Z_i > Z_j\}}_{=:~X^-} + \underbrace{\sum_{i=1}^n \textbf{1}\{Z_i < 0\}}_{=:~X^\circ} = X^+ + X^- + X^\circ.\]
This means we need to compute the terms
\begin{align*} 
\E\bigl((X_\inv^B)^2\bigr) &= \E((X^+)^2) + \E((X^-)^2) + 2\E(X^+X^-) \\
& \qquad + \E((X^\circ)^2) + 2\E(X^+X^\circ) + 2\E(X^-X^\circ).
\end{align*}
The first term $\E((X^+)^2)$ is invariant under $p,$ since it only involves events $\{Z_i > Z_j\}$ for which $\PP(Z_i > Z_j) = 1/2,$ even if the involved $Z_i, Z_j$ are not uniformly distributed. Therefore, we can obtain $\E((X^+)^2)$ from \cite[Section 3]{kahle2020counting}: 
\[\E((X^+)^2) = \frac{1}{2}\binom{n}{2} + \frac{1}{4}\binom{n}{2}\binom{n-2}{2} + \frac{5}{3}\binom{n}{3}.\]
Next, we turn to
\[\E((X^-)^2) = \sum_{i<j}\sum_{k<l} \PP(-Z_i > Z_j,-Z_k > Z_l).\]
For the $\binom{n}{2}\binom{n-2}{2}$ choices of pairwise distinct $i,j,k,l$, we have that $\PP(-Z_i > Z_j, -Z_k > Z_l) = p^2$ by independence, and for the $\binom{n}{2}$ cases where $(i,j)=(k,l)$, we simply get $\PP(-Z_i > Z_j) = p$. The set of triples with two of the indices colliding need to be analyzed similarly as in the proof of Lemma~\ref{lemma2.6}a). Note that the cases $i=k$ and $j=l$ are counted twice. E.g., in the case of $i=l,$ we calculate by case distinction according to the signs:
\begin{align*}
    \PP(-Z_i > Z_j, -Z_k > Z_i) &= \PP(-Z_i > Z_j, -Z_k > Z_i, Z_i > 0) \\
    & \qquad + \PP(-Z_i > Z_j, -Z_k > Z_i, Z_i < 0) \\
    &= \PP(-Z_i > Z_j, -Z_k > Z_i, Z_i > 0, Z_j < 0, Z_k < 0) \\
    & \qquad + \PP(Z_i < 0, Z_j < 0, Z_k < 0) \\
    & \qquad + \PP(Z_i < 0, Z_j > 0, Z_k < 0, -Z_i > Z_j) \\
    & \qquad + \PP(Z_i < 0, Z_j < 0, Z_k > 0, -Z_k > Z_i) \\
    & \qquad + \PP(Z_i < 0, Z_j > 0, Z_k > 0, -Z_i > Z_j, -Z_k > Z_i) \\
    &= \frac{1}{3}p^2q + p^3 + \frac{1}{2}p^2q + \frac{1}{2}p^2q + \frac{1}{3}pq^2 \\
    &= \frac{1}{3}p(2p+1).
\end{align*}
It turns out that all six constellations of triplets give this contribution. So, overall,
\[\E\bigl((X^-)^2\bigr) = \binom{n}{2}p + \binom{n}{2}\binom{n-2}{2}p^2 + \binom{n}{3}2p(2p+1).\]
For $\E(X^+X^-),$ the disjoint quadruplets give a contribution of $p/2$ each. For the colliding pairs, we need to compute
\begin{align*}
    &\PP(Z_i > Z_j, -Z_i > Z_j) \\
    =~& \underbrace{\PP(Z_i > Z_j, -Z_i > Z_j, Z_i > 0, Z_j > 0)}_{=~0} + \underbrace{\PP(Z_i > Z_j, -Z_i > Z_j, Z_i < 0, Z_j > 0)}_{=~0} \\
    & \quad + \PP(Z_i > Z_j, -Z_i > Z_j, Z_i < 0, Z_j < 0) + \PP(Z_i > Z_j, -Z_i > Z_j, Z_i > 0, Z_j < 0) \\
    =~& \PP(Z_i < 0, Z_j < 0, Z_i > Z_j) + \PP(Z_i > 0, Z_j < 0, -Z_i > Z_j) = \frac{p^2}{2} + p^2q.
\end{align*}
For the triplets, we repeat the procedure above. For the cases $j=k$ and $j=l$, we get
\[\PP(Z_i > Z_j, -Z_j > Z_k) = -\frac{1}{6}p(2p-5).\]
For the cases $i=l$ and $i=k$, we derive 
\begin{align*}
    \PP(Z_i > Z_j, -Z_i > Z_l) &= \PP(Z_i > Z_j, -Z_i > Z_l, Z_i > 0, Z_l < 0) \\
    & \qquad + \PP(Z_i > Z_j, -Z_i > Z_l, Z_i < 0, Z_j < 0) \\
    &= \PP(Z_i > 0, Z_j < 0, Z_l < 0, -Z_i > Z_l) \\
    & \qquad + \PP(Z_i > 0, Z_j > 0, Z_l < 0, Z_j < Z_i < -Z_l) \\
    & \qquad + \PP(Z_i < 0, Z_j < 0, Z_l < 0, Z_i > Z_j) \\
    & \qquad + \PP(Z_i < 0, Z_j < 0, Z_l > 0, Z_j < Z_i < -Z_l) \\
    &= \frac{1}{2}p^2q + \frac{1}{6}pq^2 + \frac{1}{2}p^3 + \frac{1}{6}p^2q \\
    &= \frac{1}{6}p(2p+1),
\end{align*}
and overall we obtain
\begin{align*}
    \E(X^+X^-) &= \binom{n}{2}\left(\frac{p^2}{2} + p^2q\right) + \binom{n}{2}\binom{n-2}{2}\frac{p}{2} + 3\binom{n}{3}\underbrace{\biggl(\frac{1}{6}p(2p+1) - \frac{1}{6}p(2p-5)\biggr)}_{=~ p} \\
    &= \binom{n}{2}\left(\frac{p^2}{2} + p^2q\right) + \binom{n}{2}\binom{n-2}{2}\frac{p}{2} + 3\binom{n}{3}p.
\end{align*}
The remaining three terms are easily calculated as
\begin{align*}
    \E(X^+X^\circ) &= \sum_{i<j} \sum_{k=1}^n \PP(Z_i > Z_j, Z_k < 0) \\
    &= 3\binom{n}{3}\frac{p}{2} + \sum_{i<j} \PP(Z_i > Z_j, Z_i < 0) + \sum_{i<j} \PP(Z_i > Z_j, Z_j < 0) \\
    &= 3\binom{n}{3}\frac{p}{2} + \binom{n}{2}(p^2 + pq) = 3\binom{n}{3}\frac{p}{2} + \binom{n}{2}p, \\
    \E(X^-X^\circ) &= \sum_{i<j} \sum_{k=1}^n \PP(-Z_i > Z_j, Z_k < 0) \\
    &= 3\binom{n}{3}p^2 + 2\binom{n}{2}(p^2 + p^2q) = 3\binom{n}{3}p^2 + 2\binom{n}{2}(2p^2 - p^3), \\
    \E((X^\circ)^2) &= \sum_{i,j=1}^n \PP(Z_i < 0, Z_j < 0) = 2\binom{n}{2}p^2 + np.
\end{align*}
Summing all of these terms and subtracting the square of the mean gives the claim for $B_n$. On $D_n,$ we ignore the parts involving $X^\circ$ and get the desired result. \qed

\subsection{Proof of Lemma~\ref{lemma5.5}} \label{sec9.5}

We first prove the claim on the even-signed permutation groups $D_n$, since the calculation will be the same for $B_n$. We proceed as in Lemma~\ref{lemma2.4}, starting from
\[\hat{X}_\inv = \sum_{k=1}^n \E(X_\inv \mid Z_k) - (n-1)\E(X_\inv),\]
except here, $X_\inv$ is defined by \eqref{7}, and we have $Z_k \sim \GR(p)$. According to \eqref{7}, we have
\begin{align}
    \E(X_\inv \mid Z_k) &= \sum_{i<j} \E\bigl(\textbf{1}\{Z_i > Z_j\} + \textbf{1}\{-Z_i > Z_j\} \mid Z_k\bigr) \nonumber \\
    &= \sum_{i<j} \big( \PP(Z_i > Z_j \mid Z_k) + \PP(-Z_i > Z_j \mid Z_k)\big). \label{eq8}
\end{align}
Write $f(Z_i, Z_j) := \mathbf{1}\{Z_i > Z_j\} + \mathbf{1}\{-Z_i > Z_j\}$ for $i < j$ and set $U_k = |Z_k|$. A straightforward case distinction gives

\begin{itemize*}
    \item $Z_i, Z_j > 0$: $f(Z_i, Z_j) = \textbf{1}\{U_i > U_j\},$
    \item $Z_i > 0, Z_j < 0$: $f(Z_i, Z_j) = \textbf{1}\{U_i > U_j\},$
    \item $Z_i < 0, Z_j > 0$: $f(Z_i, Z_j) = \textbf{1}\{U_i < U_j\} + 1,$ 
    \item $Z_i, Z_j < 0$: $f(Z_i, Z_j) = \textbf{1}\{U_i < U_j\} + 1$. 
\end{itemize*}
For symmetry reasons, $f(Z_i, Z_j)$ depends only on the sign of $Z_i$ but not on the sign of $Z_j$. To compute \eqref{eq8}, we only need to consider the $n-k$ tuples $(k,j)$ and the $k-1$ tuples $(i,k)$, as the other tuples are independent of $Z_k$ and produce constants which do not contribute to the variance. 

Recall that for $k < j$, we have $\PP(U_k > U_j \mid U_k) = U_k$ and $\PP(U_k < U_j \mid U_k) = 1 - U_k$. Therefore, we can write
\[\E(X_\inv \mid Z_k) = \sum_{i=1}^{k-1} \E\bigl(f(Z_i, Z_k) \mid Z_k\bigr) + \sum_{j=k+1}^n \E\bigl(f(Z_k, Z_j) \mid Z_k\bigr),\]
where 
\begin{align*}
    \E\bigl(f(Z_i, Z_k) \mid Z_k\bigr) &= \PP(Z_i > 0)(1 - U_k) + \PP(Z_i < 0)(1 + U_k) \\
    &= q(1-U_k) + p(1 + U_k) = 2pU_k - U_k+1
\end{align*}
and 
\begin{align*}
    \E\bigl(f(Z_k, Z_j) \mid Z_k\bigr) &= \textbf{1}\{Z_k > 0\}U_k + \textbf{1}\{Z_k < 0\}(1 + 1 - U_k) \\
    &= \textbf{1}\{Z_k > 0\}U_k + \textbf{1}\{Z_k < 0\}(2 - U_k).
\end{align*}
Overall, we obtain
\begin{align*}
    \E(X_\inv \mid Z_k) &= (k-1)(2p U_k - U_k + 1) \\
    &\quad + (n-k)\bigl(\textbf{1}\{Z_k > 0\}U_k + \textbf{1}\{Z_k < 0\}(2 - U_k)\bigr) + \text{const}.
\end{align*}
To use the standard formula $\Var(X) = \E(X^2) - \E(X)^2,$ where $X$ is not affected by constant summands, we now compute
\begin{subequations}
\begin{align}
    &\E(\E(X_\inv \mid Z_k)^2) \nonumber \\
    =~& (k-1)^2\E\left((2pU_k - U_k + 1)^2\right) \label{24a} \\
    +~& (n-k)^2\E\left(\bigl(\textbf{1}\{Z_k > 0\}U_k + \textbf{1}\{Z_k < 0\}(2 - U_k)\bigr)^2\right) \label{24b} \\
    +~& 2(k-1)(n-k)\E\left((2pU_k - U_k + 1)\bigl(\textbf{1}\{Z_k > 0\}U_k\right. \label{24c} \\
    & \hspace{5.9cm} + \left.\textbf{1}\{Z_k < 0\}(2 - U_k)\bigr)\right). \nonumber
\end{align}
\end{subequations}
The random variables $\textbf{1}\{Z_k > 0\}$ and $U_k$ are independent by construction and therefore,
\begin{align*}
    \eqref{24a} &= (k-1)^2\E\left(4p^2U_k^2 - 4pU_k^2 + 4pU_k + U_k^2 - 2U_k + 1\right) \\
    &= (k-1)^2\left(\frac{4}{3}p^2 - \frac{4}{3}p + 2p + \frac{1}{3}\right) = (k-1)^2\left(\frac{4}{3}p^2 + \frac{2}{3}p + \frac{1}{3}\right), \\
    \eqref{24b} &= (n-k)^2\left(q\E(U_k)^2 + p\E\bigl((2 - U_k)^2\bigr)\right) \\
    &= (n-k)^2\left(\frac{q}{3} + \frac{7}{3}p\right) = (n-k)^2\left(2p + \frac{1}{3}\right), \\
    \eqref{24c} &= 2(k-1)(n-k)\Bigl((1-p)\E(2pU_k^2 - U_k^2 + U_k) \\
    & \hspace{3cm} + p\E\bigl((2pU_k - U_k + 1)(2 - U_k)\bigr)\Bigr) \\
    &= 2(k-1)(n-k)\left((1-p)\left(\frac{2}{3}p + \frac{1}{6}\right) + p\left(\frac{4}{3}p + \frac{5}{6}\right)\right) \\
    &= 2(k-1)(n-k)\left(\frac{2}{3}p^2 + \frac{4p}{3} + \frac{1}{6}\right).
\end{align*}
In total, we have
\begin{align*}
    \E(\E(X_\inv \mid Z_k)^2) &= (k-1)^2\left(\frac{4}{3}p^2 + \frac{2}{3}p + \frac{1}{3}\right) + (n-k)^2\left(2p + \frac{1}{3}\right) \\
    &+ 2(k-1)(n-k)\left(\frac{2}{3}p^2 + \frac{4}{3}p + \frac{1}{6}\right).
\end{align*}
We subtract the square of
\begin{align*}
    \E(\E(X_\inv \mid Z_k)) &= (k-1)\left(p + \frac{1}{2}\right) + (n-k)\left(\frac{1}{2} - \frac{p}{2} + \frac{3}{2}p\right) \\
    &= \left(p + \frac{1}{2}\right)(n-1). 
\end{align*}
The variance of $\hat{X}_\inv$ on $D_n$ is
\begin{align*}
    \Var(\hat{X}_\inv) &= \sum_{k=1}^n \Big(\E(\E(X_\inv \mid Z_k)^2) - \E(\E(X_\inv \mid Z_k))^2\Big) \\
    &= \sum_{k=1}^n \E(\E(X_\inv \mid Z_k)^2) - n(n-1)^2\left(p + \frac{1}{2}\right)^2, 
\end{align*}
so, to conclude the proof, we compute
\begin{align*}
    \sum_{k=1}^n \E(\E(X_\inv \mid Z_k)^2) &= \left(\frac{4}{3}p^2 + \frac{2}{3}p + \frac{1}{3}\right)\sum_{k=1}^n (k-1)^2 + \left(2p+\frac{1}{3}\right)\sum_{k=1}^n (n-k)^2 \\
    &\qquad + \left(\frac{4}{3}p^2 + \frac{8p}{3} + \frac{1}{3}\right)\sum_{k=1}^n (k-1)(n-k) \\
    &= \left(\frac{4}{3}p^2 + \frac{2}{3}p + \frac{1}{3}\right) \cdot \frac{1}{6}n(n-1)(2n-1) \\
    &\qquad + \left(2p + \frac{1}{3}\right) \cdot \frac{1}{6}n(n-1)(2n-1) \\
    &\qquad + \left(\frac{4}{3}p^2 + \frac{8p}{3} + \frac{1}{3}\right) \cdot \frac{1}{6}n(n-1)(n-2) \\
    &= n^3\left(\frac{2}{3}p^2 + \frac{4}{3}p + \frac{5}{18}\right) - n^2\left(\frac{4}{3}p^2 - \frac{8}{3}p - \frac{1}{2}\right) \\
    & \qquad + n\left(\frac{2}{3}p^2 + \frac{4}{3}p + \frac{2}{9}\right).
\end{align*}
Subtracting $n(n-1)^2\left(p + \frac{1}{2}\right)^2$  gives exactly the desired leading term computed in Lemma~\ref{lemma5.4}. On the groups $B_n,$ we achieve the same result since the extra parts in $\Var(\hat{X}_\inv)$ yielded by $\sum_{i=1}^n \textbf{1}\{Z_i < 0\}$ are not significant. Recall that
\[X_\inv^B = X_\inv^D + \sum_{i=1}^n \textbf{1}\{Z_i < 0\},\]
and therefore, 
\begin{align*}
    \Var\left(\E(X_\inv^B \mid Z_k)\right) &= \Var\left(\E(X_\inv^D \mid Z_k) + \sum_{j=1}^n \E(\textbf{1}\{Z_j < 0\} \mid Z_k)\right) \\
    &= \Var\left(\E(X_\inv^D \mid Z_k) + \E(\textbf{1}\{Z_k < 0\} \mid Z_k) + \text{const}\right) \\
    &= \Var\left(\E(X_\inv^D \mid Z_k) + \textbf{1}\{Z_k < 0\} + \text{const}\right) \\
    &= \Var \hspace{0.7mm} \Bigl((n-k)\textbf{1}\{Z_k > 0\}U_k + (n-k)\textbf{1}\{Z_k < 0\}(2 - U_k) \\
    & \hspace{1.6cm} + (k-1)(2pU_k - U_k + 1) + \textbf{1}\{Z_k < 0\} + \text{const}\Bigr).
\end{align*}
Using the standard formula again, we have 
\begin{align*}
    \E\left(\E(X_\inv^B \mid Z_k)^2\right) &= \E\left(\E(X_\inv^D \mid Z_k)^2\right) + \E\left(\textbf{1}\{Z_k < 0\}\right) \\
    &\qquad + 2\E\left((n-k)(2-U_k)\textbf{1}\{Z_k < 0\}\right) \\
    &\qquad + 2\E\bigl((k-1)(2pU_k - U_k + 1)\textbf{1}\{Z_k < 0\}\bigr) \\
    &= \E\left(\E(X_\inv^D \mid Z_k)^2\right) + p + 3p(n-k) + (k-1)(2p+1)p, \\
    \E\left(\E(X_\inv^B \mid Z_k)\right)^2 &= \E\left(\E(X_\inv^D \mid Z_k)\right)^2 + p^2 + 2\E(\textbf{1}\{Z_k < 0\})\E\left(\E(X_\inv^D \mid Z_k)\right) \\
    &= \left(p + \frac{1}{2}\right)^2(n-1)^2 + p^2 + 2p(n-1)\left(p + \frac{1}{2}\right).
\end{align*}
We conclude that 
\begin{align*}
    \Var(\hat{X}_\inv^B) &= \Var(\hat{X}_\inv^D) + n(p - p^2) - p(2p+1)\frac{n(n-1)}{2} \\
    &= \Var(\hat{X}_\inv^D) + O(n^2).
\end{align*}
The desired claim follows from Theorem~\ref{thm2.2}. \qed

\subsection{Proof of Lemma~\ref{lemma5.6}a)} \label{sec9.6}

For the groups $B_n$ and $D_n,$ the calculation follows the same approach as on $S_n$. Recall that now, $Z_1, \ldots, Z_n \sim \GR(p).$ On $D_n,$ we have by \eqref{14b}, \eqref{7} that 
\begin{subequations}
\begin{align}
    \Cov(X_\inv^D, X_\des^D) &= \sum_{i<j} \sum_{k=1}^{n-1} \Cov\left(\textbf{1}\{Z_i > Z_j\}, \textbf{1}\{Z_k > Z_{k+1}\}\right) \label{21a} \\
    &+ \sum_{i<j} \sum_{k=1}^{n-1} \Cov\left(\textbf{1}\{-Z_i > Z_j\}, \textbf{1}\{Z_k > Z_{k+1}\}\right)  \label{21b} \\
    &+ \sum_{i<j} \Cov(\textbf{1}\{Z_i > Z_j\} + \textbf{1}\{-Z_i > Z_j\}, \textbf{1}\{-Z_2 > Z_1\}). \label{21c}
\end{align}
\end{subequations}
The contribution of \eqref{21a} is $(n-1)/4,$ in analogy with Lemma~\ref{lemma2.6}a). In \eqref{21b}, we first demonstrate the cancellation in the non-exceptional case when $\{i - 1, i, j - 1, j\}$ form a set of distinct numbers. In that case, we have
\begin{align*}
    & \Cov(\textbf{1}\{-Z_i > Z_j\}, \textbf{1}\{Z_i > Z_{i+1}\}) + \Cov(\textbf{1}\{-Z_i > Z_j\}, \textbf{1}\{Z_{i-1} > Z_i\}) \\
    =~& \E(\textbf{1}\{-Z_i > Z_j\}\textbf{1}\{Z_i > Z_{i+1}\}) + \E(\textbf{1}\{-Z_i > Z_j\}\textbf{1}\{Z_{i-1} > Z_i\}) - 2(p/2) \\
    =~& \E(\textbf{1}\{-Z_i > Z_j\}\textbf{1}\{Z_i > Z_{i+1}\}) + \E(\textbf{1}\{-Z_i > Z_j\}\textbf{1}\{Z_{i+1} > Z_i\}) - p \\
    =~& \E(\textbf{1}\{-Z_i > Z_j\}) - p = 0,
\end{align*}
and likewise, 
\[\Cov(\textbf{1}\{-Z_i > Z_j\}, \textbf{1}\{Z_j > Z_{j+1}\}) + \Cov(\textbf{1}\{-Z_i > Z_j\}, \textbf{1}\{Z_{j-1} > Z_j\}) = 0.\]
However, this cancellation occurs not only in the non-exceptional cases, but also in the aggregation of the exceptional cases (E1) -- (E6) from the proof of Lemma~\ref{lemma2.6}a), except for the covariances resulting from the clash of $j=i+1$ and $k=i$. To be precise, (E2) and (E3) cancel mutually. (E4) and (E5) give two clashes and a canceling pair. (E6) consists of another canceling pair. All of this can be checked from the computation of $\E(X^+X^-)$ in the proof of Lemma~\ref{lemma5.4}. From there, we also obtain $\forall i=1, \ldots, n-1$:
\[\Cov(\textbf{1}\{-Z_i > Z_{i+1}\}, \textbf{1}\{Z_i > Z_{i+1}\}) = \frac{p^2}{2} + p^2q - \frac{p}{2},\]
which interestingly vanishes in the unbiased case. Overall,
\[\eqref{21b} = (n-1)\left(\frac{p^2}{2} + p^2q - \frac{p}{2}\right).\]
Finally, consider \eqref{21c}. Obviously, this double-indexed sum involves exactly the pairs $(i,j)$ with $i=1$ or $i=2$. We get 
\begin{align*}
    \eqref{21c} &= \sum_{j=3}^n \Cov(\textbf{1}\{Z_1 > Z_j\} + \textbf{1}\{-Z_1 > Z_j\}, \textbf{1}\{-Z_2 > Z_1\}) \\
    &+ \sum_{j=3}^n \Big[\Cov(\textbf{1}\{Z_2 > Z_j\} + \textbf{1}\{-Z_2 > Z_j\}, \textbf{1}\{-Z_2 > Z_1\}) \\
    &+ \underbrace{\Cov(\textbf{1}\{-Z_1 > Z_2\}, \textbf{1}\{-Z_2 > Z_1\})}_{= \hspace{0.7mm} p^2} + \underbrace{\Cov(\textbf{1}\{Z_1 > Z_2\}, \textbf{1}\{-Z_2 > Z_1\})}_{= \hspace{0.7mm} 0}\Big] \\
    &= \underbrace{\sum_{j=3}^n \Big[\Cov(\textbf{1}\{Z_1 > Z_j\}, \textbf{1}\{-Z_2 > Z_1\}) + \Cov(\textbf{1}\{-Z_1 > Z_j\}, \textbf{1}\{-Z_2 > Z_1\})\Big]}_{= \hspace{0.7mm} 0} \\
    &+ \underbrace{\sum_{j=3}^n \Big[\Cov(\textbf{1}\{Z_2 > Z_j\}, \textbf{1}\{-Z_2 > Z_1\}) + \Cov(\textbf{1}\{-Z_2 > Z_j\}, \textbf{1}\{-Z_2 > Z_1\})\Big]}_{= \hspace{0.7mm} 0} + \hspace{0.7mm} p^2.
\end{align*}
Therefore, we obtain the overall result on $D_n$, which is
\[\Cov(X_\inv^D, X_\des^D) = (n-1)\left(\frac{p^2}{2} + p^2q - \frac{p}{2} + \frac{1}{4}\right) + p^2.\]
Next, we show that this result is obtained on $B_n$ as well. By \eqref{6} and \eqref{14a}, we have
\begin{subequations}
\begin{align}
    \Cov(X_\inv^B, X_\des^B) &= \eqref{21a} + \eqref{21b} \nonumber \\
    &+ \sum_{i=1}^n \sum_{k=1}^{n-1} \Cov(\textbf{1}\{Z_i < 0\}, \textbf{1}\{Z_k > Z_{k+1}\}) \label{22a} \\
    &+ \sum_{i<j} \Cov(\textbf{1}\{Z_i > Z_j\} + \textbf{1}\{-Z_i > Z_j\}, \textbf{1}\{Z_1 < 0\}) \label{22b} \\
    &+ \sum_{i=1}^n \Cov(\textbf{1}\{Z_i < 0\}, \textbf{1}\{Z_1 < 0\})\,. \label{22c}
\end{align}
\end{subequations}
We now see that \eqref{22a} and \eqref{22b} vanish. In \eqref{22a}, the inner sum only involves $k=i-1$ and $k=i,$ therefore, 
\begin{align*}
    \eqref{22a} &= \sum_{i=1}^n \Big[ \Cov(\textbf{1}\{Z_i < 0\}, \textbf{1}\{Z_{i-1} > Z_i\}) + \Cov(\textbf{1}\{Z_i < 0\}, \textbf{1}\{Z_i > Z_{i+1}\})\Big] \\
    &= \sum_{i=1}^n \Big[\Cov(\textbf{1}\{Z_i < 0\}, \textbf{1}\{Z_{i-1} > Z_i\}) + \Cov(\textbf{1}\{Z_i < 0\}, \textbf{1}\{Z_i > Z_{i-1}\}) \Big] \\
    &= \E(\textbf{1}\{Z_1 < 0\}) - 2p/2 = 0.
\end{align*}
This cancellation even applies to the boundary terms $i=1$ and $i=n$. We also get 
\begin{align*}
    \eqref{22b} &= \sum_{j=2}^n \Big[ \Cov(\textbf{1}\{Z_1 > Z_j\}, \textbf{1}\{Z_1 < 0\}) + \Cov(\textbf{1}\{-Z_1 > Z_j\}, \textbf{1}\{Z_1 < 0\})\Big] \\
    &= \underbrace{\sum_{j=2}^n \Big[\PP(Z_1 < 0, Z_1 > Z_j) - p/2\Big]}_{= \hspace{0.7mm} -pq/2} + \underbrace{\sum_{j=2}^n \Big[\PP(Z_1 < 0, -Z_1 > Z_j) - p^2\Big]}_{= \hspace{0.7mm} pq/2} = 0.
\end{align*}
Finally,
\[\eqref{22c} = \sum_{i=1}^n \Cov(\textbf{1}\{Z_i < 0\}, \textbf{1}\{Z_1 < 0\}) = \Var(\textbf{1}\{Z_1 < 0\}) = p - p^2,\]
giving the overall result
\[\Cov(X_\inv^B, X_\des^B) = (n-1)\left(\frac{p^2}{2} + p^2q - \frac{p}{2} + \frac{1}{4}\right) + (p - p^2), \]
so again, $\rho(X_\inv^B, X_\des^B)$ vanishes in the limit. \qed

\subsection{Proof of Lemma~\ref{lemma5.6}b)} \label{sec9.7}

For the groups $B_n$ and $D_n$, with $Z_k \sim \GR(p)$ and the modifications \eqref{14a}, \eqref{14b}, the calculation is more extensive but its procedure is the same as in the proof of Lemma~\ref{lemma2.6}b). On $D_n$, we first have
\begin{subequations}
\begin{align}
    \Cov(\hat{X}_\inv^D, X_\des^D) = \sum_{j=1}^n\sum_{k=1}^{n-1} &(n-j)\Cov(U_j\textbf{1}\{Z_j > 0\}, \textbf{1}\{Z_k > Z_{k+1}\}) \label{23a} \\
    +~&(n-j)\Cov((2-U_j)\textbf{1}\{Z_j < 0\}, \textbf{1}\{Z_k > Z_{k+1}\}) \label{23b} \\
    +~& (j-1)\Cov(2pU_j - U_j + 1, \textbf{1}\{Z_k > Z_{k+1}\}). \label{23c} \\
    + \sum_{j=1}^n &(n-j)\Cov(U_j\textbf{1}\{Z_j > 0\}, \textbf{1}\{-Z_2 > Z_1\}) \label{23d} \\
    +~&(n-j)\Cov((2-U_j)\textbf{1}\{Z_j < 0\}, \textbf{1}\{-Z_2 > Z_1\}) \label{23e} \\
    +~& (j-1)\Cov(2pU_j - U_j + 1, \textbf{1}\{-Z_2 > Z_1\}). \label{23f}
\end{align}
\end{subequations}
In the first three rows \eqref{23a} -- \eqref{23c}, there is cancellation if $j \notin \{1,n\}$ due to previously used arguments. Only $j=1$ is relevant in \eqref{23a}, \eqref{23b} and only $j=n$ is relevant in \eqref{23c}. We have $\E(U_j\textbf{1}\{Z_j > 0\})\E(\textbf{1}\{Z_j > Z_{j+1}\}) = q/4$, and the joint density of $Z_j$ and $Z_{j+1}$ is
\[f_p(x,y) := f_p(x)f_p(y) = \begin{cases} p^2, & x, y < 0 \\ pq, & x>0, y<0 \\ pq, & x<0, y>0 \\ q^2, & x,y>0 \end{cases}.\]
By Fubini's Theorem, we obtain
\begin{align*}
    \E(U_1&\textbf{1}\{Z_1 > 0\}\textbf{1}\{Z_1 > Z_2\}) = \int_{[-1,1]^2} |x|\textbf{1}\{x>0\}\textbf{1}\{x>y\}f_p(x,y)\text{d}(x,y) \\
    &= \int_{[0,1]^2} q^2x\textbf{1}\{x>y\}\text{d}(x,y)  + \int_{[0,1]\times [-1,0]} pqx\textbf{1}\{x>y\}\text{d}(x,y) \\
    &= q^2\int_0^1 x^2dx + pq\int_0^1 x\dx = \frac{q^2}{3} + \frac{pq}{2},
\end{align*}
and likewise,
\begin{align*}
    \E\bigl((2 - U_1)\textbf{1}\{Z_1 < 0\}\textbf{1}\{Z_1 > Z_2\}\bigr) &= \int_{-1}^0 \int_{-1}^1 (2 - |x|)\textbf{1}\{x>y\}f_p(x,y)\text{d}(x,y) \\
    &= p^2\int_{-1}^0 (2+x)(1+x)\dx = \frac{5}{6}p^2, \\
    \E\bigl((2p-1)U_n\textbf{1}\{Z_{n-1}>Z_n\}\bigr) &= (2p-1)\int_{[-1,1]^2} |x|\textbf{1}\{x<y\}f_p(x,y)\text{d}(x,y) \\
    &= (2p-1)\left(\frac{p^2}{3} + \frac{q^2}{6} + \frac{pq}{2}\right) = (2p-1)\frac{p+1}{6}.
\end{align*}
Therefore, 
\[\eqref{23a} + \eqref{23b} + \eqref{23c} = (n-1)\left(-\frac{p^3}{3} + \frac{5p^2}{3} - \frac{4p}{3} + \frac{1}{6}\right)\]
has linear order. The remaining three rows \eqref{23d} -- \eqref{23f} also have no more than linear order, since the summands are nonzero only for $j=1$ or $j=2$. \hfill \qed

\bibliographystyle{acm}
\bibliography{bibliography.bib}

\begin{thebibliography}{10}

\bibitem{arratia2003logarithmic}
{\sc Arratia, R., Barbour, A.~D., and Tavar{\'e}, S.}
\newblock {\em Logarithmic combinatorial structures: a probabilistic approach}, vol.~1.
\newblock European Mathematical Society, 2003.

\bibitem{arslan2018unfair}
{\sc Arslan, {\.I}., I{\c{s}}lak, {\"U}., and Pehlivan, C.}
\newblock On unfair permutations.
\newblock {\em Statistics \& Probability Letters 141\/} (2018), 31--40.

\bibitem{bender1973central}
{\sc Bender, E.~A.}
\newblock Central and local limit theorems applied to asymptotic enumeration.
\newblock {\em Journal of Combinatorial Theory, Series A 15}, 1 (1973), 91--111.

\bibitem{bjorner2006combinatorics}
{\sc Björner, A., and Brenti, F.}
\newblock {\em Combinatorics of {C}oxeter groups}, vol.~231.
\newblock Springer Science \& Business Media, 2006.

\bibitem{bruck2019central}
{\sc Br{\"u}ck, B., and R{\"o}ttger, F.}
\newblock A central limit theorem for the two-sided descent statistic on {C}oxeter groups.
\newblock {\em Electronic Journal of Combinatorics 29}, 1 (2022), P1.1.

\bibitem{chang2021central}
{\sc Chang, J., Chen, X., and Wu, M.}
\newblock Central limit theorems for high dimensional dependent data.
\newblock {\em Bernoulli 30}, 1 (2024), 712--742.

\bibitem{chatterjee2017central}
{\sc Chatterjee, S., and Diaconis, P.}
\newblock A central limit theorem for a new statistic on permutations.
\newblock {\em Indian Journal of Pure and Applied Mathematics 48}, 4 (2017), 561--573.

\bibitem{chernozhukov2013gaussian}
{\sc Chernozhukov, V., Chetverikov, D., and Kato, K.}
\newblock Gaussian approximations and multiplier bootstrap for maxima of sums of high-dimensional random vectors.
\newblock {\em The Annals of Statistics 41}, 6 (2013), 2786--2819.

\bibitem{chernozhukov2017central}
{\sc Chernozhukov, V., Chetverikov, D., and Kato, K.}
\newblock Central limit theorems and bootstrap in high dimensions.
\newblock {\em The Annals of Probability 45}, 4 (2017), 2309--2352.

\bibitem{chernozhukov2023nearly}
{\sc Chernozhukov, V., Chetverikov, D., and Koike, Y.}
\newblock Nearly optimal central limit theorem and bootstrap approximations in high dimensions.
\newblock {\em The Annals of Applied Probability 33}, 3 (2023), 2374--2425.

\bibitem{conger2007normal}
{\sc Conger, M., and Viswanath, D.}
\newblock Normal approximations for descents and inversions of permutations of multisets.
\newblock {\em Journal of Theoretical Probability 20}, 2 (2007), 309--325.

\bibitem{conger2005refinement}
{\sc Conger, M.~A.}
\newblock A refinement of the {E}ulerian numbers, and the joint distribution of {$\pi(1)$} and {${\rm Des}(\pi)$} in {$S_n$}.
\newblock {\em Ars Combin. 95\/} (2010), 445--472.

\bibitem{das2021central}
{\sc Das, D., and Lahiri, S.}
\newblock Central limit theorem in high dimensions: The optimal bound on dimension growth rate.
\newblock {\em Transactions of the American Mathematical Society 374}, 10 (2021), 6991--7009.

\bibitem{dorr2022extreme}
{\sc D{\"o}rr, P., and Kahle, T.}
\newblock Extreme values of permutation statistics.
\newblock {\em Electronic Journal of Combinatorics\/} (2024).
\newblock Forthcoming.

\bibitem{fang2021high}
{\sc Fang, X., and Koike, Y.}
\newblock {High-dimensional central limit theorems by Stein’s method}.
\newblock {\em The Annals of Applied Probability 31}, 4 (2021), 1660--1686.

\bibitem{fang2015rates}
{\sc Fang, X., and R{\"o}llin, A.}
\newblock Rates of convergence for multivariate normal approximation with applications to dense graphs and doubly indexed permutation statistics.
\newblock {\em Bernoulli 21}, 4 (2015), 2157--2189.

\bibitem{fulman2004stein}
{\sc Fulman, J.}
\newblock Stein's method and non-reversible markov chains.
\newblock {\em Lecture Notes-Monograph Series\/} (2004), 69--77.

\bibitem{harper1967stirling}
{\sc Harper, L.~H.}
\newblock Stirling behavior is asymptotically normal.
\newblock {\em Ann. Math. Statist. 38\/} (1967), 410--414.

\bibitem{janson1988normal}
{\sc Janson, S.}
\newblock Normal convergence by higher semiinvariants with applications to sums of dependent random variables and random graphs.
\newblock {\em The Annals of Probability\/} (1988), 305--312.

\bibitem{kahle2020counting}
{\sc Kahle, T., and Stump, C.}
\newblock Counting inversions and descents of random elements in finite {C}oxeter groups.
\newblock {\em Mathematics of Computation 89}, 321 (2020), 437--464.

\bibitem{koike2021notes}
{\sc Koike, Y.}
\newblock Notes on the dimension dependence in high-dimensional central limit theorems for hyperrectangles.
\newblock {\em Japanese Journal of Statistics and Data Science 4}, 1 (2021), 257--297.

\bibitem{meier2022central}
{\sc Meier, K., and Stump, C.}
\newblock Central limit theorems for generalized descents and generalized inversions in finite root systems.
\newblock {\em Electron. J. Probab. 28\/} (2023), Paper No. 125, 25.

\bibitem{pike2011convergence}
{\sc Pike, J.}
\newblock Convergence rates for generalized descents.
\newblock {\em The Electronic Journal of Combinatorics\/} (2011), P236--P236.

\bibitem{pitman1997probabilistic}
{\sc Pitman, J.}
\newblock Probabilistic bounds on the coefficients of polynomials with only real zeros.
\newblock {\em Journal of Combinatorial Theory, Series A 77}, 2 (1997), 279--303.

\bibitem{sibuya1960bivariate}
{\sc Sibuya, M., et~al.}
\newblock Bivariate extreme statistics.
\newblock {\em Annals of the Institute of Statistical Mathematics 11}, 2 (1960), 195--210.

\bibitem{van2000asymptotic}
{\sc van~der Vaart, A.~W.}
\newblock {\em Asymptotic statistics}, vol.~3.
\newblock Cambridge university press, 2000.

\end{thebibliography}

\end{document}